\documentclass{amsart}
\usepackage{amssymb}
\usepackage[all]{xy}
\usepackage{ifthen}
\usepackage{color}
\usepackage{tikz-cd}
\usepackage[colorlinks,anchorcolor=blue,citecolor=blue,linkcolor=blue,urlcolor=blue,bookmarksdepth=2]{hyperref}
\newboolean{workmode}
\setboolean{workmode}{false}
\newboolean{sections}
\setboolean{sections}{true}
\usepackage{comment}
\usepackage{enumitem}
\usepackage{mathtools}
\newtheorem{thm}{Theorem}[section]
\newtheorem{cor}[thm]{Corollary}
\newtheorem{lem}[thm]{Lemma}
\newtheorem{prop}[thm]{Proposition}

\theoremstyle{definition}
\newtheorem{defn}[thm]{Definition}
\newtheorem{exam}[thm]{Example}

\theoremstyle{remark}
\newtheorem{rem}[thm]{Remark}
\numberwithin{equation}{section}

\newcommand{\sK}{{\mathcal K}}

\newcommand{\sO}{{\mathcal O}}

\newcommand{\A}{{\mathbb A}}

\newcommand{\C}{{\mathbb C}}

\newcommand{\E}{{\mathbb E}}
\newcommand{\F}{{\mathbb F}}
\newcommand{\G}{{\mathbb G}}

\newcommand{\N}{{\mathbb N}}

\newcommand{\Z}{{\mathbb Z}}

\newcommand{\x}{\xrightarrow}
\newcommand{\fp}{{\mathfrak p}}

\newcommand{\fq}{{\mathfrak q}}

\newcommand{\SL}{{\rm SL}}

\newcommand{\Spec}{{\rm Spec}}
\newcommand{\End}{{\rm End}}

\newcommand{\Hom}{\operatorname{Hom}}

\title[Galois representations of products of Drinfeld modules]{Galois representation of the product of two Drinfeld modules of generic characteristic}
\author[L. Duan and J. Fang]{Lian Duan \and Jiangxue Fang\\
\text{With an appendix by Xuanyou Li}}

\address{Institute of Mathematical Sciences\\
  ShanghaiTech University\\
  No.393 Middle Huaxia Road, Pudong New District,
Shanghai, China}\par\nopagebreak
\email{duanlian@shanghaitech.edu.cn}

\address{Department of Mathematics,
Capital Normal University, Beijing 100148, P.R. China} \email{fangjiangxue@gmail.com}

\date{}

\begin {document}
\topmargin= -.2in \baselineskip=18pt

\begin{abstract}
We study the adelic Galois representations attached to products of two Drinfeld modules of generic characteristic. As a function-field analogue of Serre's open image theorem for products of elliptic curves, we show that the image of the joint adelic representation is commensurable with the subgroup defined by the natural determinant compatibility conditions if and only if at least one of the two Drinfeld modules has potential complex multiplication, or the two modules are not geometrically isogenous up to any Frobenius twist. Our proof combines Pink's theory of minimal quasi-models, explicit reciprocity laws from global class field theory, and an analysis of residual Galois representations.

As an arithmetic application, we prove a torsion finiteness result: if one of the Drinfeld modules has no potential complex multiplication and the two are not geometrically isogenous up to any Frobenius twist, then that Drinfeld module has only finitely many torsion points over the compositum of the maximal abelian extension of the base field and the torsion field of the other module. This is an analogue of the classical results of Ribet and Zarhin for abelian varieties.
\end{abstract}

\maketitle
\markboth{\shorttitle}{\shorttitle}
\pagestyle{headings}

\tableofcontents

\section{Introduction}\label{Sect: introduction}

A natural question in arithmetic geometry asks \emph{to what extent a Galois representation reflects the geometry and arithmetic of the underlying variety?} In his seminal works \cite{Se68,Se72}, Serre established foundational results regarding this dependence for elliptic curves over number fields. He first proved that for a single elliptic curve without potential complex multiplication, the image of its adelic Galois representation is an open subgroup of ${\rm GL}_2(\widehat{\mathbb Z})$. Furthermore, Serre demonstrated that for the product of two non-isogenous elliptic curves, their joint Galois representations are as independent as possible: the image of the adelic representation attached to their product is an open subgroup of the product of their respective adelic groups, subject only to a natural determinant compatibility condition. Serre’s work may be regarded as one of the first steps toward the study of the so-called Mumford--Tate conjecture (see Moonen \cite{Mo} for further background).

Later, Pink \cite{P} and Pink-R\"{u}tsche \cite{PR} established the corresponding function field analogue for a single Drinfeld module. More precisely, they proved that if a Drinfeld module without potential complex multiplication is defined over a finitely generated field, then the image of its associated adelic Galois representation is likewise open. This naturally raises the primary question of this paper: \emph{What happens for the product of two Drinfeld modules?}

We now briefly recall the relevant framework. Let $p$ be a prime number, and let $q$ be a power of $p$. Let $C$ be a geometrically connected smooth projective curve over the finite field $\F_q$, and fix a closed point $\infty$ of $C$. Let
$
A=\Gamma(C-\{\infty\},\sO_C)
$
be the ring of regular functions on $C$ with poles only at $\infty$. Let $F$ denote the function field of $C$. Let $K/F$ be a finitely generated field extension, and fix an algebraic closure $\overline{K}$ of $K$. We write $K^{\mathrm{sep}}$ for the separable closure of $K$ inside $\overline{K}$. Throughout this paper, we write
$
\Gamma_K\coloneqq{\rm Gal}(K^{\mathrm{sep}}/K)
$
for the absolute Galois group of $K$.

Let $\phi$ be a Drinfeld $A$-module over $K$ of rank $n$. For a maximal ideal $\mathfrak p$ of $A$, let $A_{\mathfrak p}$ and $F_{\mathfrak p}$ denote the completions of $A$ and $F$ at $\mathfrak p$, and let $T_{\mathfrak p}\phi$ be the $\mathfrak p$-adic Tate module of $\phi$ (see \S\,\ref{Sect: arithmetic}), which carries a continuous action of $\Gamma_K$. We write
\[
  \rho_{\mathfrak p}:\Gamma_K\to\mathrm{GL}_{A_{\mathfrak p}}(T_{\mathfrak p}\phi),
  \qquad
  \rho:\Gamma_K\to\prod_{\mathfrak p\in{\rm Max}\,A}\mathrm{GL}_{A_{\mathfrak p}}(T_{\mathfrak p}\phi)
\]
for the resulting $\mathfrak p$-adic and adelic representations, where ${\rm Max}\,D$ denotes the set of maximal ideals of a ring $D$. After replacing $K$ by a finite extension, the action of $\Gamma_K$ commutes with ${\rm End}_{\overline K}(\phi)$, so $\rho_{\mathfrak p}(\Gamma_K)$ lies in $\mathrm{GL}_{A_{\mathfrak p}\otimes_A{\rm End}_{\overline K}(\phi)}(T_{\mathfrak p}\phi)$.

Note that two subgroups $H$ and $N$ in a group are said to be \emph{commensurable} if $H\cap N$ has finite index in both $H$ and $N$. Pink and R\"{u}tsche proved the following open image theorem.

\begin{thm}\cite[Theorem 0.2]{PR}\label{Thm: Pink_open_image}
 $\rho(\Gamma_K)$ and $\prod\limits_{\mathfrak p\in{\rm Max}\,A}\mathrm{GL}_{A_{\mathfrak{p}} \otimes_{A} \mathrm{End}_{\overline{K}}(\phi)}(T_{\mathfrak{p}}\phi)$ are commensurable in $\prod\limits_{\mathfrak p\in{\rm Max}\,A}{\rm GL}_{A_\mathfrak p}(T_\mathfrak p\phi)$.
\end{thm}

\subsection{Statement of the main theorem}\label{Sect: main_thm}
Motivated by this classical picture, our goal is to establish a corresponding independence result for pairs of Drinfeld modules. Roughly speaking, we prove that the image of the joint Galois representation attached to two Drinfeld modules is sufficiently large precisely when the expected geometric obstruction is absent: either one of the two modules has potential complex multiplication, or the two modules are not geometrically isogenous up to any Frobenius twist (see Definition~\ref{defnfrobtwist}). We refer the reader to Theorem~\ref{thmmain} for the precise formulation.

To state our results precisely, take another Drinfeld module $\phi'$ over $K$ of rank $n'$. Let $\rho'_{\mathfrak p}$ and $\rho'$ denote the associated $\mathfrak p$-adic and adelic representations, respectively. We study the image of the product representations
\begin{equation}\label{eqnvarrho}
  \begin{aligned}
    &(\rho_{\mathfrak p},\rho'_{\mathfrak p}):\Gamma_K\to{\rm GL}_{A_\mathfrak p}(T_\mathfrak p\phi)\times{\rm GL}_{A_\mathfrak p}(T_\mathfrak p\phi'),\\
   &(\rho,\rho'):\Gamma_K\to\prod_{\mathfrak p\in{\rm Max}\,A}{\rm GL}_{A_\mathfrak p}(T_\mathfrak p\phi)\times{\rm GL}_{A_\mathfrak p}(T_\mathfrak p\phi').
  \end{aligned}
\end{equation}

Let $V_{\mathfrak{p}}\phi = F_{\mathfrak{p}} \otimes_{A_{\mathfrak{p}}} T_{\mathfrak{p}}\phi$, $E = F \otimes_A \mathrm{End}_{\overline{K}}(\phi)$, and $E_{\mathfrak{p}} = E \otimes_F F_{\mathfrak{p}}$. Then $\mathrm{End}_{\overline{K}}(\phi)$ is commutative, and $E$ is a finite field extension of $F$. Moreover, $V_\fp\phi$ is a free $E_\fp$-module of rank $\frac{n}{[E:F]}$ (\cite[Proposition 4.7.6, Corollary 4.7.15, Proposition 4.7.19]{Go}). For any $g_{\mathfrak p}\in{\rm GL}_{A_\mathfrak p\otimes_A{\rm End}_{\overline K}(\phi)}(T_\mathfrak p\phi)$, we define $\det_{E_{\mathfrak p}}(g_{\mathfrak p})$ as the determinant of $g_{\mathfrak p}$ viewed as an element of ${\rm GL}_{E_{\mathfrak p}}(V_\mathfrak p\phi)$.

The natural inclusion $\mathrm{End}_{\overline{K}}(\phi)\hookrightarrow\overline{K}\{\tau\}$ defines a tautological Drinfeld $\mathrm{End}_{\overline{K}}(\phi)$-module, which we denote by $\psi$. The rank of $\psi$ over $\mathrm{End}_{\overline{K}}(\phi)$ is $\frac{n}{[E:F]}$ (see more about $\psi$ in \S\,\ref{Sect: simplifications}). The homomorphism ${\rm End}_{\overline K}(\phi)\to \overline K$ sending an element to its constant term induces a homomorphism $\nu:E\to\overline K$.  Similarly, we define the corresponding objects $\phi'$, $\psi'$, $V_{\mathfrak{p}}\phi'$, $E'$, $E'_{\mathfrak{p}}$, $\det_{E'_{\mathfrak p}}(g'_{\mathfrak p})$, and $\nu'$. In this way, we obtain a commutative diagram of fields
\begin{equation}\label{eqnF1F2}
  \xymatrix{F\ar[d]\ar[r]&E\ar[d]_{\nu}\\E'\ar[r]^{\nu'}&\overline K.}
\end{equation}

Set $F'=E\cap E'$, $F'_{\mathfrak p}=F'\otimes_F F_{\mathfrak p}$ and let ${\rm N}_{E_{\mathfrak p}/F'_{\mathfrak p}}$ (resp. ${\rm N}_{E'_{\mathfrak p}/F'_{\mathfrak p}}$) be the norm map from $E_{\mathfrak p}$ (resp. $E'_{\mathfrak p}$) to $F'_{\mathfrak p}$. Define
\begin{equation}\label{Eqn: U_p}
  U_\mathfrak p\coloneqq\{(g_{\mathfrak p}, g'_{\mathfrak p})|{\rm N}_{E_{ \mathfrak p}/F'_{\mathfrak p}}({\rm det}_{E_{ \mathfrak p}}(g_{ \mathfrak p}))
    =
  {\rm N}_{E'_\mathfrak p/F'_{\mathfrak p}}({\rm det}_{E'_{ \mathfrak p}}(g'_{ \mathfrak p}))\}
\end{equation}
to be a subgroup of $ {\rm GL}_{A_\mathfrak p\otimes_A{\rm End}_{\overline K}(\phi)}(T_\mathfrak p\phi)\times{\rm GL}_{A_\mathfrak p\otimes_A{\rm End}_{\overline K}(\phi')}(T_\mathfrak p\phi')$. 
Let $U=\prod\limits_{\mathfrak p\in{\rm Max}\,A}U_\mathfrak p$.

With these notations, our main result is the following. We refer the reader to Definition~\ref{defnfrobtwist} for the definitions of a Frobenius twist.

\begin{thm}\label{thmmain}
 The following three statements are equivalent:
  \begin{itemize}
    \item[(1)] One of $\phi$ and $\phi'$ has potential complex multiplication, or $\phi$ and $\phi'$ are not geometrically isogenous up to any Frobenius twist.
    \item[(2)] There exists $\fp\in {\rm Max}\,A$ such that $(\rho_\fp,\rho'_\fp)(\Gamma_K)$ and $U_\fp$ are commensurable in ${\rm GL}_{A_\fp}(T_\fp\phi)\times{\rm GL}_{A_\fp}(T_\fp\phi')$.
    \item[(3)] $(\rho,\rho')(\Gamma_K)$ and $U$ are commensurable in $\prod\limits_{\fp\in {\rm Max}\,A}{\rm GL}_{A_\fp}(T_\fp\phi)\times{\rm GL}_{A_\mathfrak p}(T_\mathfrak p\phi')$.
  \end{itemize}
\end{thm}

The main difficulty in proving Theorem~\ref{thmmain} comes from the positive characteristic setting, and it appears in three places. First, in characteristic zero the study of Galois images can often be reduced to $\ell$-adic Lie groups and their Lie algebras, and this method is not available in the same form here. Second, if two Drinfeld modules differ by a Frobenius twist, the image of their joint Galois representation is commensurable with the graph of an isomorphism between the two factors. This phenomenon has no analogue in characteristic zero, and it makes it subtle to identify the correct algebraic envelope of the image and to prove the desired openness statements. Third, to pass from local openness at a fixed prime to simultaneous commensurability at infinitely many primes, one must control the congruences between the mod-$\mathfrak{p}$ representations of the two modules and rule out exceptional coincidences for almost all $\mathfrak{p}$. To overcome these difficulties, we use Pink's theory of minimal quasi-models \cite{PC}, which shows that a compact subgroup of the rational points of a product of adjoint algebraic groups is, up to isogeny, controlled by an algebraic group defined over a uniquely determined closed subfield. Applied to the adjoint representations attached to $\phi$ and $\phi'$, it yields the required openness at finitely many primes. We treat the determinant representations by explicit class field theory together with openness results for norm maps between idele groups. Finally, to pass from local to adelic openness, we show that exceptional relations at infinitely many pairs of primes would force the resulting congruences between Frobenius invariants to lift to an isomorphism of the global endomorphism fields. This in turn gives a relation between the local representations and hence a geometric isogeny up to a Frobenius twist, contradicting the hypothesis. The local and residual information is then combined by a profinite-group criterion for openness in infinite products.

\subsection{Connections with related work}

\subsubsection{Connection to the Mumford--Tate group of pairs}

The study of independence of Galois representations attached to products of abelian varieties is closely related to the Mumford--Tate conjecture. In this direction, Lombardo \cite{Lom} showed that, under certain additional assumptions, the Mumford--Tate group of a product of two non-isogenous abelian varieties is essentially determined by those of the individual factors, subject to suitable compatibility conditions. In particular, under these assumptions, the Mumford--Tate conjecture for each factor implies the conjecture for their product. This implication was subsequently strengthened by Commelin \cite{Com}, who removed the additional assumptions.

Our work establishes a function-field analogue of this independence phenomenon for pairs of Drinfeld modules. Thanks to the favorable structural properties of Drinfeld modules, we obtain this result without imposing additional hypotheses analogous to those required in Lombardo's work.

\subsubsection{Torsion finiteness of Drinfeld modules}
We also highlight another motivation for this project. The classical Mordell--Weil theorem asserts that an abelian variety defined over a number field has only finitely many rational torsion points. Ribet \cite[Appendix]{KLR} proved that any abelian variety $A$ defined over a number field has only finitely many torsion points over the maximal abelian extension $\mathbb{Q}^{\rm ab}$ of $\mathbb{Q}$. This result was later generalized by Zarhin \cite{Zar} by showing that an abelian variety defined over a number field $K$ has finitely many torsion points over the maximal abelian extension $K^{\rm ab}$, provided that $A$ does not have complex multiplication over $K$.

By the Kronecker-Weber theorem, Ribet's result could be restated as $A$ has only finitely many torsion points defined over the field generated by the torsion points of the algebraic torus $\mathbb{G}_m$. Observed this, one expects that two geometrically non-isogenous abelian varieties should be mutually ``torsion finite''---meaning one variety has only finitely many torsion points defined over the field generated by the torsion points of the other. Indeed, a straightforward Galois-theoretic argument shows that Serre's theorem implies this mutual torsion finiteness for pairs of elliptic curves, a phenomenon recently extended to certain abelian varieties by Achter, Wang, and the first author \cite{ADW}.

Motivated by this geometric disparity, it is natural to expect a similar torsion finiteness property for non-isogenous Drinfeld modules. In fact, this expectation motivates part~(3) of Theorem~\ref{thmmain}.



As a consequence of our main theorem (Theorem~\ref{thmmain}), we obtain the following mutual torsion finiteness result for non-isogenous Drinfeld modules, the proof of which is provided in \S\,\ref{Sect: final_pf}.

\begin{thm}\label{Cor: torsion_finite}
  Let $\phi$ and $\phi'$ be two Drinfeld modules as above. Suppose that $\phi$ does not have potential complex multiplication, and that $\phi$ and $\phi'$ are not geometrically isogenous up to any Frobenius twist. Then $\phi$ has only finitely many torsion points defined over the compositum $K^{\rm ab}K_{\phi'}$, where $K_{\phi'}$ denotes the extension of $K$ obtained by adjoining all torsion points of $\phi'(K^{\rm sep})$, and $K^{\rm ab}$ is the maximal abelian extension of $K$. 
\end{thm}

\subsection{Structure of this paper}

In \S\,\ref{Sect: simplifications}, we state the technical reductions that allow us to assume suitable endomorphism and integrality properties of $\phi$ and $\phi'$, and in \S\,\ref{Sect: arithmetic} we recall the necessary background on Drinfeld modules. Section~\ref{Sect: adj_repn} treats the adjoint representations via Pink's minimal quasi-model theory, Section~\ref{Sect: determinant} treats the determinant representations via class field theory, and Section~\ref{Sect: img_res_repn} studies the residual representations. Section~\ref{Sect: final_pf} proves Theorem~\ref{thmmain} and deduces the torsion finiteness result, Theorem~\ref{Cor: torsion_finite}. Appendix~A proves a coincidence theorem for morphisms of schemes over finite fields, which is used in Section~\ref{Sect: adj_repn}.

\subsection{Some simplifications}\label{Sect: simplifications}
To streamline this paper and avoid repetition, we adopt the following convention.

While our analysis involves two distinct Drinfeld modules $\phi$ and $\phi'$, most properties can be established uniformly without distinguishing between them. To simplify our discussion, we introduce the notation $\phi^\flat$, where $\flat$ denotes either $\emptyset$ or ${}'$. Thus, $\phi^\flat$ represents either $\phi$ or $\phi'$ as needed. This convention extends naturally to all objects associated with $\phi$ and $\phi'$.

Note that Theorem~\ref{thmmain} holds for $K$ if it holds for some finite extension of $K$, thus we will always assume
\begin{itemize}
  \item[] \textbf{ Simplification 1}: ${\rm End}_{\overline K}(\phi^\flat)={\rm End}_K(\phi^\flat)$.
\end{itemize}

Recall that $\phi^\flat$ factors through the natural inclusion  $\psi^\flat:{\rm End}_{K}(\phi^\flat)\to K\{\tau\}$
(${\rm End}_{K}(\phi^\flat)$ is commutative \cite[Proposition~4.7.6]{Go}).
Let $E^\flat=F\otimes_A{\rm End}_{K}(\phi^\flat)$, which is a finite extension of $F$.
Note that $\psi^\flat|_A=\phi^\flat$, so the underlying Drinfeld $A$-module structure is unchanged.

Let $B^\flat$ be the integral closure of ${\rm End}_{K}(\phi^\flat)$ in $E^\flat$. By \cite[Proposition~1.1]{P}, there exists a Drinfeld $B^\flat$-module
$
\varphi^\flat:B^\flat\to K\{\tau\}
$
which is isogenous to $\psi^\flat$ when restricted to ${\rm End}_{K}(\phi^\flat)$. That is, there exists $g\in K\{\tau\}$ such that
$
g\,\psi^\flat(x)=\varphi^\flat(x)\,g \quad \text{for all } x\in{\rm End}_{K}(\phi^\flat).
$
The induced isogeny $g:\phi^\flat\to\varphi^\flat|_A$ yields isomorphisms
$
T_\mathfrak p\phi^\flat \simeq T_\mathfrak p(\varphi^\flat|_A)
$
for almost all $\mathfrak p$. Consequently, Theorem~\ref{thmmain} holds for $\phi^\flat$ if it holds for $\varphi^\flat|_A$.
Replacing the former by the latter and noting that
$
{\rm End}_K(\varphi^\flat|_A)=B^\flat,
$
we obtain the second simplification.

\begin{itemize}
  \item[] \textbf{ Simplification 2}: ${\rm End}_K(\phi^\flat)$ is integrally closed in its fraction field $E^\flat$.
\end{itemize}

Under Simplification~2, we have $B^\flat={\rm End}_K(\phi^\flat)$, and the tautological homomorphism
$
\psi^\flat:B^\flat\to K\{\tau\}
$
is a Drinfeld $B^\flat$-module over $K$ of rank $\frac{n^\flat}{[E^\flat:F]}$. By \cite[Theorem~1.8]{P}, there exists a Drinfeld $B^\flat$-module
$
\det(\psi^\flat):B^\flat\to K\{\tau\}
$
of rank one, together with a $\Gamma_K$-equivariant isomorphism
$
T_{\mathfrak q^\flat}\det(\psi^\flat)\simeq\bigwedge^{\frac{n^\flat}{[E^\flat:F]}} T_{\mathfrak q^\flat}\psi^\flat
$
for any maximal ideal $\mathfrak q^\flat$ of $B^\flat$. We call $\det(\psi^\flat)$ the \emph{determinant} of $\psi^\flat$.

Let $A'$ be the integral closure of $A$ in $F'=E\cap E'$. Then the restriction
$
\det(\psi^\flat)|_{A'}:A'\to K\{\tau\}
$
is a Drinfeld $A'$-module of rank $[E^\flat:F']$, and its determinant $\det(\det(\psi^\flat)|_{A'})$ is a Drinfeld $A'$-module of rank one. Note that, after possibly replacing $K$ by a finite extension, any two rank-one Drinfeld $A'$-modules over $K$ are isogenous, and each of the rank-one modules occurring here admits a model over a finite extension of $F'$. Hence, for the purpose of proving Theorem~\ref{thmmain}, we may impose the following third simplification.
\begin{itemize}
  \item[] \textbf{ Simplification 3}: The two Drinfeld $A'$-modules $\det(\det(\psi)|_{A'})$ and $\det(\det(\psi')|_{A'})$ are isogenous and can be defined over some finite subextension of $F$ in $K$.
\end{itemize}

\begin{rem}\label{Ufp}
  Under these three simplifications, for any maximal ideal $\fp$ of $A$, we have $(\rho_\fp,\rho'_\fp)(\Gamma_K)\subset U_\fp$ and $U_\fp$ is the subgroup of $\prod\limits_{\fq\mid\fp}{\rm GL}_{B_\fq}(T_\fq\psi)\times\prod\limits_{\fq'\mid\fp}{\rm GL}_{B'_{\fq'}}(T_{\fq'}\psi')$
  consisting of those elements $(g_\fq,g_{\fq'})$ such that for any maximal ideal $\fp'$ of $A'$ above $\fp$, we have $$\prod\limits_{\fq\mid\fp'}{\rm N}_{E_\fq/F'_{\fp'}}(\det(g_\fq))=\prod\limits_{\fq'\mid\fp'}{\rm N}_{E'_{\fq'}/F'_{\fp'}}(\det(g'_{\fq'})).$$
\end{rem}

\section{Arithmetic of Drinfeld modules}\label{Sect: arithmetic}
In this section, we recall some basic facts about the arithmetics of Drinfeld modules. For details, we refer to Goss \cite{GL,Go}, Lehmkuhl \cite{Le}, and Pink \cite{P}. Let
\[
  \phi: A\to K\{\tau\}, \quad a\mapsto \phi_a
\]
be a Drinfeld $A$-module over $K$ of rank $n$. 

\subsection{Tate modules}

For any nonzero ideal $\mathfrak a$ of $A$, define
\[
  \phi[\mathfrak a]=\{x\in K^{\rm sep}\mid \phi_a(x)=0 \text{ for all } a\in\mathfrak a\}.
\]
Then $\phi[\mathfrak a]$ is a free $A/\mathfrak a$-module of rank $n$. For any maximal ideal $\mathfrak p$ of $A$, set
\[
  \phi[\mathfrak p^\infty]=\bigcup_{d\geq1}\phi[\mathfrak p^d].
\]
The \emph{Tate modules} $T_\mathfrak p\phi$ and rational Tate module $V_\mathfrak p\phi$ of $\phi$ at $\mathfrak p$ are defined by
\begin{eqnarray*}
  T_\mathfrak p\phi&\coloneqq&{\rm Hom}_A\Big(F_\mathfrak p/A_\mathfrak p,\phi[\mathfrak p^\infty]\Big),\\
  V_\mathfrak p\phi&\coloneqq&{\rm Hom}_A\Big(F_\mathfrak p,\phi[\mathfrak p^\infty]\Big)= F_{\mathfrak{p}} \otimes_{A_{\mathfrak{p}}} T_{\mathfrak{p}}\phi.
\end{eqnarray*}
Then $T_\mathfrak p\phi$ is a free $A_\mathfrak p$-module of rank $n$, and $V_\mathfrak p\phi$ is an $F_\mathfrak p$-vector space of dimension $n$.
\subsection{Frobenius twist}
\begin{defn}\label{defnfrobtwist}
  For any $d\in\N$, the $d$-th \emph{Frobenius twist}  of $\phi$ is defined to be the homomorphism
  \begin{eqnarray*}
    \phi^{(d)}:A&\to &\overline K\{\tau\}\\a&\mapsto&\Big(\sum_ia_i^{\frac{1}{p^d}}\tau^i\Big)^{p^d}
  \end{eqnarray*}
  if $\phi_a=\sum\limits_ia_i\tau^i$.

  We say two Drinfeld modules $\phi$ and $\phi'$ are \emph{isogenous up to a Frobenius twist} if there exists a natural number $d$ such that $\phi^{(d)}$ is isogenous to $\phi'$ or  $\phi$ is isogeneous to $\phi'^{(d)}$.
\end{defn}

\begin{rem}
  The Frobenius twist $\phi^{(d)}$ may not be defined over $K$, since $K\{\tau\}$ is noncommutative. However, it is always defined over a finite extension of $K$.
\end{rem}

For any $\F_p$-algebra $D$, denote ${\rm Fr}_D$ to be the $p$-th power map on $D$. For any $D$-module $M$, its $d$-th twist $M^{(d)}$ is the $D$-module with the same underlying abelian group as $M$, but with $D$-action given by
\[
  D\times M\to M,\quad (a,m)\mapsto a^{p^d}m.
\]
\begin{lem}\label{lemfrobtwist}
  Let $d$ be a natural number. Then $\phi^{(d)}$ and $\phi\circ{\rm Fr}^d_A$ have the same rank $np^d$. Let $K'$ be a finite extension of $K$ over which $\phi^{(d)}$ is defined. Then there are ${\rm Gal}(\overline K/K')$-equivariant $F_\mathfrak p$-linear isomorphisms
  \[
    V_\mathfrak p(\phi^{(d)})\simeq V_\mathfrak p(\phi\circ{\rm Fr}^d_A)\simeq (V_\mathfrak p\phi)^{(d)}.
  \]
\end{lem}

\begin{proof}
  The first isomorphism follows from the bijection $x\mapsto x^{p^d}:\phi^{(d)}[\fp^\infty]\to(\phi\circ{\rm Fr}^d_A)[\fp^\infty]$ and the second one from the bijection $x\mapsto x:(\phi\circ{\rm Fr}^d_A)[\fp^\infty]\to\phi[\fp^\infty]$.  The statement on the rank follows directly from the definition.
\end{proof}

\begin{lem}\label{lem:frobtwist-endomorphisms}
Assume Simplifications~1 and~2, and suppose that $\phi'$ is geometrically isogenous to $\phi^{(d)}$ for some $d\geq0$. Then there exists a field isomorphism
\[
  \sigma:E\xrightarrow{\sim}E'
\]
such that
\[
  \sigma(B)=B',
  \qquad
  \sigma(a^{p^d})=a \quad(\forall a\in A).
\]
Consequently, for every $\fq\in{\rm Max}\,B$, the prime
$\fq'=\sigma(\fq)$ lies above the same prime of $A$ as $\fq$,
and $\sigma$ extends to an isomorphism
\[
  \widehat\sigma:E_\fq\xrightarrow{\sim}E'_{\fq'}.
\]
After a finite extension of $K$, there is a
$\Gamma_K$-equivariant $\widehat\sigma$-semilinear isomorphism
\[
  V_\fq\psi\xrightarrow{\sim}V_{\fq'}\psi'.
\]
\end{lem}

\begin{proof}
Fix an isogeny $h:\phi^{(d)}\to\phi'$. Replace $K$ with its finite extension if necessary, we can assume that $\phi^{(d)}$ and all its endomorphisms are defined over $K$.

Define the ring automorphism
\[
  R_d:\overline K\{\tau\}\longrightarrow\overline K\{\tau\},
  \qquad
  \sum_i c_i\tau^i\longmapsto\sum_i c_i^{1/p^d}\tau^i.
\]
By the definition of the Frobenius twist,
\[
  \phi^{(d)}_a=R_d(\phi_{a^{p^d}})
  \qquad(a\in A).
\]

We first observe that
\[
  \End_{\overline K}(\phi\circ{\rm Fr}_A^d)
  =
  \End_{\overline K}(\phi).
\]
The inclusion from right to left is immediate. For the converse, every $\phi_a$ belongs to the endomorphism ring on the left. Since this ring is commutative, every element of it commutes with every $\phi_a$, proving the reverse inclusion. Therefore
\[
  \End_{\overline K}(\phi^{(d)})
  =
  R_d\bigl(\End_{\overline K}(\phi)\bigr).
\]
Taking fraction fields and then conjugating by the quasi-isogeny $h$, we obtain an isomorphism
\[
  \sigma:E\xrightarrow{\sim}E',
  \qquad
  u\longmapsto hR_d(u)h^{-1}.
\]

Under the structural embeddings of $A$ into $E$ and $E'$, the element $a$ is represented by $\phi_a$ and $\phi'_a$,
respectively. Since
\[
  hR_d(\phi_{a^{p^d}})h^{-1}
  =
  h\phi^{(d)}_a h^{-1}
  =
  \phi'_a,
\]
we have $\sigma(a^{p^d})=a$ for every $a\in A$.

By Simplification~2, $B$ is the integral closure of $A$ in $E$. Since $A$ is integral over $A^{p^d}$, it is also the integral closure of $A^{p^d}$ in $E$. As $\sigma(A^{p^d})=A$, the isomorphism $\sigma$ identifies $B$ with the integral closure of $A$ in $E'$, namely $B'$.

For $\fq'=\sigma(\fq)$ and $a\in A$, we have
\[
  a\in\fq'
  \quad\Longleftrightarrow\quad
  \sigma^{-1}(a)=a^{p^d}\in\fq
  \quad\Longleftrightarrow\quad
  a\in\fq.
\]
Thus $\fq$ and $\fq'$ lie above the same prime of $A$, and hence $\sigma$ extends to the stated isomorphism of completions, say $\widehat{\sigma}$.

Finally, identifying $B$ and $B'$ with their endomorphisms, the construction gives
\[
  hR_d(b)=\sigma(b)h
  \qquad(b\in B).
\]
Consequently, the map on torsion points
\[
  x\longmapsto h(x^{1/p^d})
\]
is $\sigma$-semilinear and $\Gamma_K$-equivariant.
Taking rational Tate modules at the corresponding primes
gives the required $\widehat\sigma$-semilinear isomorphism.
\end{proof}

\subsection{Specialization}\label{Sect: specialization}

In this subsection, $\mathfrak p$ is a fixed maximal ideal of $A$. Since $A$ is finitely generated over $\F_q$, there exists a finitely generated $\F_q$-subalgebra $R$ of $K$ satisfying the following conditions:
\begin{itemize}
  \item[(a)] The fraction field of $R$ is $K$, and $R$ is integrally closed.
  \item[(b)] For any nonzero $a\in A$, all coefficients of $\phi_a$ lie in $R$, and the leading coefficient of $\phi_a$ is a unit in $R$.
  \item[(c)] The image of the structure morphism
    $x\mapsto \mathfrak p_x:{\rm Spec}\,R \longrightarrow {\rm Spec}\,A$
    is contained in ${\rm Spec}\,A\setminus\{\mathfrak p\}$.
\end{itemize}
Then $\phi$ defines a Drinfeld module $\phi:A\to R\{\tau\}$ over $X\coloneqq{\rm Spec}\,R$. For each closed point $x\in X$, we denote by $\phi_x$\footnote{This notation should not be confused with $\phi_a$, which denotes the image of $a\in A$ under $\phi$.} the induced Drinfeld module over the residue field $\kappa(x)$ of $x$.

Let $\widetilde R$ be the integral closure of $R$ in $K^{\rm sep}$ and let $\widetilde x$ be a closed point of ${\rm Spec}\;\widetilde R$ above $x$. As $\mathfrak p_x\neq\mathfrak p$, there are natural isomorphisms $$\phi(K^{\rm sep})[\mathfrak p^\infty]=\phi(\widetilde R)[\mathfrak p^\infty]\simeq\phi_{x}(\overline{\kappa(x)})[\mathfrak p^\infty].$$
This induces an isomorphism $V_\mathfrak p\phi\simeq V_\mathfrak p\phi_x$ of Tate modules and a commutative diagram
\[\xymatrix{D(\widetilde x)\ar[r]^{\rho_{\mathfrak p}}\ar[d]&{\rm GL}_{F_\mathfrak p}(V_\mathfrak p\phi)\ar[d]^\simeq\\
{\rm Gal}(\overline{\kappa(x)}/\kappa(x))\ar[r]&{\rm GL}_{F_\mathfrak p}(V_{\mathfrak p}\phi_{x}),}\]
where $D(\widetilde x)=\{\gamma\in\Gamma_K\mid \gamma(\widetilde x)=\widetilde x\}$ is the decomposition group of $\widetilde x$.

Let ${\rm Frob}_x$ denote the Frobenius element of ${\rm Gal}(\overline{\kappa(x)}/\kappa(x))$. Then the characteristic polynomial
$
\det\bigl(t-{\rm Frob}_x,\;V_{\mathfrak p}\phi\bigr)
$
is well defined and independent of the choice of $\widetilde x$. 

For a matrix $g$ over a field, write
\begin{eqnarray}\label{eqn:a(g)}
a(g)={\rm Tr}(g)\cdot{\rm Tr}(g^{-1}).\end{eqnarray}
This function is invariant under scalar multiplication, conjugation, and inverse transpose. We thus have a function
\begin{eqnarray*}
  |X|\to F_\mathfrak p,\qquad x\mapsto a(\rho_\fp({\rm Frob}_x))={\rm Tr}({\rm Frob}_x,\,V_\mathfrak p\phi)\cdot{\rm Tr}({\rm Frob}_x^{-1},\,V_\mathfrak p\phi).
\end{eqnarray*}
The rank of the Tate module of $\phi_x$ at $\mathfrak p_x$ is less than that of $\phi$ at $\mathfrak p$. We define the deficiency
\[
  h_x\coloneqq n-\dim_{F_{\mathfrak p_x}}\bigl(V_{\mathfrak p_x}\phi_x\bigr)
\]
to be the \emph{height} of $\phi_x$.

\begin{thm}\label{2}
  Let $\overline{F}$ be the algebraic closure of $F$ in $\overline K$. Then the following statements hold:
  \begin{itemize}
    \item[(1)] \cite[Theorem 3.2.3(b)]{GL} For any $x\in|X|$, the characteristic polynomial $\det(t-{\rm Frob}_x,\;V_{\mathfrak p}\phi)$ has coefficients in $A$ and is independent of $\mathfrak p$.

    \item[(2)] \cite[Theorem 3.2.3(d)]{GL} For any $x\in|X|$, the ideal $\mathfrak p_x^{[\kappa(x):\kappa({\mathfrak p_x})]}$ is principal, generated by $\det({\rm Frob}_x,\,V_\mathfrak p\phi)$.

    \item[(3)] \cite[Proposition 2.4]{P} If $n\geq2$, and suppose $A=\End_{\overline{K}}(\phi)$, then as a field, $F$ is generated by the elements $a(\rho_\fp({\rm Frob}_x))$ for all $x\in|X|$.

    \item[(4)] \cite[Theorem~1.3]{P} For any $x\in |X|$, let $v_{\bar{\fp}_x}$ be a valuation of $\overline{F}$ extending the discrete valuation of $F$ at $\fp_x$. Let $\{\alpha_i\}_{i=1}^n$ be the roots of $\det(t-{\rm Frob}_x,\;V_{\mathfrak p}\phi)$ in $\overline F$. Then
      \[
        v_{\bar{\fp}_x}(\alpha_i)=
        \begin{cases}
          >0 & \text{for precisely $h_x$ of the $\alpha_i$},\\
          =0 & \text{for the remaining roots}.
        \end{cases}
      \]
  \end{itemize}
\end{thm}

\section{Adjoint Galois representation of two Drinfeld modules}\label{Sect: adj_repn}
In this section, we study the projectivized Galois representations associated with $\psi$ and $\psi'$. Our main tool is Pink's theory of minimal quasi-models \cite{PC}, which we recall in \S\,\ref{Sect: Min. mod. }. In \S\,\ref{Sect: adj_repn_2_ideals}, we apply this theory to the joint Galois representation at a pair of maximal ideals $(\fq,\fq')\in{\rm Max}\,B\times{\rm Max}\,B'$. The key result is Theorem~\ref{11}, which shows that a nontrivial relation between the corresponding projective representations forces $\phi$ and $\phi'$ to be geometrically isogenous up to a Frobenius twist. We conclude the section with Theorem~\ref{thmmainsl}.

\subsection{Minimal model theory of compact subgroups of linear algebraic groups}\label{Sect: Min. mod. }
In this subsection, we briefly review Pink's minimal model theory. For further details, we refer to \cite{PC}.

We fix the following notation. Let $S$ be an arbitrary finite index set. For each $i \in S$, let $L_i$ be a local field and let $G_i$ be an absolutely simple adjoint algebraic group over $L_i$, i.e.\ a geometrically simple algebraic group whose adjoint representation is faithful. Set $L_S = \prod\limits_{i \in S} L_i$ and let $G_S = \coprod\limits_{i \in S} G_i$ be the disjoint union viewed as a scheme. Then $G_S$ can be regarded as a group scheme over $L_S$ with
$
G_S(L_S) = \prod\limits_{i \in S} G_i(L_i),
$
so that each $G_i$ identifies with the fiber of $G_S$ over $L_i$. Let $\Gamma_S \subset G_S(L_S)$ be a compact subgroup that is fiberwise Zariski dense; namely, for each $i \in S$, the image $\Gamma_i$ of $\Gamma_S$ in $G_i(L_i)$ is Zariski dense in $G_i$.

\begin{defn}\cite[Definition 3.1, 3.2, 3.4, 3.5]{PC}
  (1) A \emph{weak quasi-model} of $(L_S, G_S,\Gamma_S)$ is a triple $(L,H,\varphi)$, where
  \begin{itemize}
    \item[(i)] $L$ is a semisimple closed subring of $L_S$, and $L_S$ is of finite type as module over $L$.
    \item[(ii)] $H$ is a fiberwise absolutely simple adjoint algebraic group over $L$.
    \item[(iii)] $\varphi$ is an isogeny $H\times_L L_S\to G_S$ of group schemes such that $\Gamma_S$ is contained in the subgroup $\varphi(H(L))$ of $G_S(L_S)$.
  \end{itemize}

  (2) A weak quasi-model $(L,H,\varphi)$ of $(L_S, G_S,\Gamma_S)$ is called a \emph{quasi-model} if the derivative of $\varphi$ vanishes nowhere.

  (3) We call that $(L_S, G_S,\Gamma_S)$ to be \emph{minimal} if for any of its weak quasi-model $(L,H,\varphi)$, we have $L=L_S$ and $\varphi$ is an isomorphism.

  (4) A (weak) quasi-model $(L,H,\varphi)$ of $(L_S, G_S,\Gamma_S)$ is called minimal if $(L,H,\varphi^{-1}(\Gamma_S))$ is minimal in its own right.
\end{defn}

\begin{prop}\cite[Proposition~3.9]{PC}\label{Prop: proj_minimal}
  For any nonempty subset $T$ of $S$, let $L_T = \prod\limits_{i \in T} L_i$, $G_T = \coprod\limits_{i \in T} G_i$, and let $\Gamma_T$ be the image of $\Gamma_S$  under the projection $G_S(L_S)\twoheadrightarrow G_T(L_T)$.  If $(L_S, G_S, \Gamma_S)$ is minimal, then so is $(L_T, G_T, \Gamma_T)$.
\end{prop}

\begin{defn}\label{defngamma_S} Let $\widetilde G_S$ be the universal covering of $G_S$. Since the natural isogeny $\widetilde{G}_S\to G_S$ is central, the commutator morphism
  $$[\;,\;]^\sim :\widetilde{G}_S\times \widetilde{G}_S\to \widetilde{G}_S,\;(g,h)\mapsto ghg^{-1}h^{-1}$$
  factors through
  $$[\;,\;] :G_S\times G_S\to\widetilde G_S.$$
  Denote by $\widetilde{\Gamma}_S$ the closed subgroup of $\widetilde G_S(L_S)$ generated by $[\;,\;](\Gamma_S\times\Gamma_S)$.
\end{defn}
The main theorem of \cite{PC} is the following.
\begin{thm}\label{thmminimalmodel}\cite[Theorem 3.6, Proposition 7.1, Corollary 7.3]{PC}
  (1) There exists a minimal quasi-model $(L,H,\varphi)$ of $(L_S, G_S,\Gamma_S)$.

  (2) The subring $L$ of $L_S$ is unique, and $H$ and $\varphi$ are determined up to a unique isomorphism.

  (3) $(L_S,G_S,\Gamma_S)$ is minimal if and only if  $\widetilde{\Gamma}_S$ is open in $\widetilde G_S(L_S)$.
\end{thm}

\subsection{Adjoint representation of \texorpdfstring{$\psi$ and $\psi'$}{psi and psi'} at two maximal ideals}\label{Sect: adj_repn_2_ideals}
Recall that $\psi^\flat \colon B^\flat = \mathrm{End}_{K}(\phi^\flat) \to K\{\tau\}$ is the Drinfeld modules defined in \S\,\ref{Sect: simplifications}. For any $\fq^\flat\in{\rm Max}\,B^\flat$, we have the $\mathfrak q^{\flat}$-adic Galois representation and its adjoint representation associated to $\psi^\flat$:
$$\rho^\flat_{\fq^\flat}:\Gamma_K\to{\rm GL}_{B^{\flat}_{\mathfrak q^{\flat}}}(T_{\mathfrak q^{\flat}}\psi^{\flat}),\qquad\bar\rho^\flat_{\fq^\flat}:\Gamma_K\to{\rm PGL}_{E^{\flat}_{\mathfrak q^{\flat}}}(V_{\mathfrak q^{\flat}}\psi^{\flat}).$$
In this subsection, fix a pair $(\fq,\fq')\in{\rm Max}\,B\times{\rm Max}\,B'$. We study the the product representation
$$(\bar\rho_\fq,\bar\rho'_{\fq'}):\Gamma_K\to{\rm GL}_{B_{\fq}}(T_{\fq}\psi)\times{\rm GL}_{B'_{\fq'}}(T_{\fq'}\psi'),$$
by using Pink's minimal model theory.

\begin{lem}\label{3}
  Let $S=\{\fq, \fq'\}$ and let $(L,H,\varphi)$ be the minimal quasi-model of the triple 
  $$(L_S, G_S, \Gamma_S)=\Big(E_\fq\times E'_{\fq'},\;{\rm PGL}_{E_{\fq},V_{\fq}\psi}\coprod{\rm PGL}_{E'_{\fq'},V_{\fq'}\psi'},\; (\bar\rho_\fq,\bar\rho'_{\fq'})(\Gamma_K)\Big).$$ Then one of the following two cases holds:
  \begin{itemize}
    \item[(1)] $L=E_{\mathfrak q}\times E'_{\mathfrak q'}$.
    \item[(2)] There exists an isomorphism $\sigma:E_{\mathfrak q}\simeq E'_{\mathfrak q'}$ such that $L=\{(x,\sigma(x))\mid x\in E_{\mathfrak q}\}$.
  \end{itemize}
\end{lem}
\begin{proof}


Let $L_{\mathfrak q^{\flat}}$ be the image of $L$ in $E^{\flat}_{\mathfrak q^{\flat}}$, and let $\Gamma_{\fq^{\flat}}=\bar\rho^{\flat}_{\mathfrak q^{\flat}}(\Gamma_K)$. By Theorem~\ref{Thm: Pink_open_image}, $\widetilde{\Gamma}_{\fq^{\flat}}$ is open in the topological group ${\rm SL}_{E^{\flat}_{\mathfrak q^{\flat}}}(V_{\mathfrak q^{\flat}}\psi^{\flat})$. Set $H_{\mathfrak q^{\flat}}=H\times_LL_{\mathfrak q^{\flat}}$, and let $\varphi_{\mathfrak q^{\flat}}$ be the composition
$$H_{\mathfrak q^{\flat}}\times_{L_{\mathfrak q^{\flat}}}E^{\flat}_{\mathfrak q^{\flat}}=(H\times_LL_{\mathfrak q^{\flat}})\times_{L_{\mathfrak q^{\flat}}}E^{\flat}_{\mathfrak q^{\flat}}=(H\times_LL_S)\times_{L_S} E^{\flat}_{\mathfrak q^{\flat}}\x{\varphi\times{\rm id}_{E^{\flat}_{\mathfrak q^{\flat}}}}G_S\times_{L_S}E^{\flat}_{\mathfrak q^{\flat}}={\rm PGL}_{E^{\flat}_{\mathfrak q^{\flat}},V_{\mathfrak q^{\flat}}\psi^{\flat}}.$$
By Proposition~\ref{Prop: proj_minimal}, $(L_{\mathfrak q^{\flat}},H_{\mathfrak q^{\flat}},\varphi_{\mathfrak q^{\flat}})$ is a minimal quasi-model of $(E^{\flat}_{\mathfrak q^{\flat}},{\rm PGL}_{E^{\flat}_{\mathfrak q^{\flat}},V_{\mathfrak q^{\flat}}\psi^{\flat}},\Gamma_{\mathfrak q^{\flat}})$. On the other hand, according to Theorem~\ref{thmminimalmodel} (3), $(E^{\flat}_{\mathfrak q^{\flat}},{\rm PGL}_{E^{\flat}_{\mathfrak q^{\flat}},V_{\mathfrak q^{\flat}}\psi^{\flat}},\Gamma_{\mathfrak q^{\flat}})$ is already minimal due to the openness of $\widetilde{\Gamma}_{\fq^{\flat}}$ in ${\rm SL}_{E^{\flat}_{\mathfrak q^{\flat}}}(V_{\mathfrak q^{\flat}}\psi^{\flat})$. Thus $E^{\flat}_{\mathfrak q^{\flat}}=L_{\mathfrak q^{\flat}}$, and $\varphi_{\mathfrak q^{\flat}}:H_{\mathfrak q^{\flat}}\times_{L_{\mathfrak q^{\flat}}}E^{\flat}_{\mathfrak q^{\flat}}\to {\rm PGL}_{E^{\flat}_{\mathfrak q^{\flat}},V_{\mathfrak q^{\flat}}\psi^{\flat}}$ is an isomorphism.
In other words, $L$ maps onto both $E_{\mathfrak q}$ and $E'_{\mathfrak q'}$. The lemma then follows immediately from the commutative ring version of Goursat's lemma.
\end{proof}

Recall that Lemma~\ref{lem:frobtwist-endomorphisms} shows that a geometric isogeny up to a Frobenius twist induces a graph relation between the corresponding local projective Galois representations. Our key theorem establishes the converse. That is, if the minimal quasi-model at a pair $(\fq,\fq')$ is defined over the graph of a local field isomorphism, then $\phi$ and $\phi'$ are geometrically isogenous up to a Frobenius twist.

Before stating the theorem, note that under Simplification~1 in \S\,\ref{Sect: simplifications}, the commutative diagram~\eqref{eqnF1F2} takes the form
\[
  \xymatrix{F\ar[r]\ar[d]&E\ar[d]^\nu\\E'\ar[r]^{\nu'}&K.}
\]

\begin{thm}\label{11}
  Keep the notations and assumptions in Lemma~\ref{3}, and assume that $L$ to be the graph of an isomorphism $\sigma:E_{ \mathfrak q }\to E'_{ \mathfrak q' }$, i.e., the second case of Lemma~\ref{3} holds. Suppose $\psi$ and $\psi'$ both have rank $\geq2$.

  (1) Then $\sigma(E)=E',\sigma(B)=B'$ and $\sigma(\mathfrak q)=\mathfrak q'$.

  (2) There exists $d\in\N$ making one of the following two diagrams
  \begin{eqnarray}\label{eqncomdiag}\xymatrix{F\ar[dd]\ar[rr]&&E\ar[dd]^\nu\ar[ld]_{{\rm Fr}_E^d}\\&E\ar[ld]_{\sigma|_E}&\\E'\ar[rr]^{\nu'}&&K}\quad\quad\quad\xymatrix{F\ar[dd]\ar[rr]&&E\ar[dd]^\nu\\&E\ar[ld]_{\sigma|_E}\ar[ru]^{{\rm Fr}_E^d}&\\E'\ar[rr]^{\nu'}&&K}
  \end{eqnarray}
  commutes.

  (3) The Drinfeld modules $\phi$ and $\phi'$ are isogenous up to a Frobenius twist.
\end{thm}
\begin{proof}
  We prove this theorem in four steps. In Step~(i), we prove that 
  $$a(\rho'_{\fq'}(\gamma))=\sigma(a(\rho_\fq(\gamma)))\qquad(\gamma\in\Gamma_K),$$ where the function $a$ is defined in (\ref{eqn:a(g)}). 
  In Step~(ii), we prove that $\sigma(E)=E'$. In Step~(iii), we show that the induced isomorphism $\sigma|_E:E\simeq E'$ makes one of the diagrams in \eqref{eqncomdiag} commutative for some $d$. In Step~(iv), we prove that $\phi$ and $\phi'$ are isogenous up to a Frobenius twist.

  \underline{Step (i).} Prove $a(\rho'_{\fq'}(\gamma))=\sigma(a(\rho_\fq(\gamma)))$ for any $\gamma\in\Gamma_K$.

  Let $\bar\eta$ denote the composition of isomorphisms of algebraic groups
  $${\rm PGL}_{E_{\mathfrak q},V_{\mathfrak q}\psi}\x{\varphi_{\mathfrak q}^{-1}}H_{\mathfrak q}\simeq H\times_LE_{\mathfrak q}\x{{\rm id}_H\times\sigma}H\times_LE'_{\mathfrak q'}\simeq H_{\mathfrak q'}\x{\varphi_{\mathfrak q'}}{\rm PGL}_{E'_{\mathfrak q'},V_{\mathfrak q'}\psi'}.$$
  This shows that $\frac n{[E:F]}=\frac{n'}{[E':F]}$. Put $m=\frac n{[E:F]}\geq 2$ to be the common rank of $\psi$ and $\psi'$. Then there exists a $\sigma$-semilinear isomorphism
  \begin{equation}\label{Eqn: sigma_semilinear_f}
    f:V_{\mathfrak q}\psi\to V_{\mathfrak q'}\psi'\hbox{ or }f:(V_{\mathfrak q}\psi)^\vee\to V_{\mathfrak q'}\psi'
  \end{equation}
  such that $\bar\eta$ is given by either $g\mapsto f\circ g\circ f^{-1}$ or  $g\mapsto f\circ (g^\vee)^{-1}\circ f^{-1}$. Here we use the notation $V^\vee$ to denote the dual space of a vector space $V$ over a field, and denote by $g^\vee$ the dual of a linear map $g$. 
  
  Note that every automorphism of ${\rm PGL}_{E'_{\mathfrak q'},V_{\mathfrak q'}\psi'}$ is inner if $m=2$. 
  As $(\bar\rho_\fq,\bar\rho'_{\fq'})(\Gamma_K)\subset\varphi(H(L))$, there exists a character $\chi:\Gamma_K\to (E'_{ \mathfrak q' })^\times$ such that one of the following two cases holds:
  \begin{eqnarray}  \label{6}
    &&\rho'_{ \mathfrak q' }(\gamma)=\chi(\gamma)\cdot (f\circ \rho_{ \mathfrak q }(\gamma^{-1})^\vee\circ f^{-1})\quad\quad (\gamma\in\Gamma_K), \quad \text{if }m\geq 3,
    \\\label{9}
    &&\rho'_{ \mathfrak q' }(\gamma)=\chi(\gamma)\cdot (f\circ \rho_{ \mathfrak q }(\gamma)\circ f^{-1})\quad\quad (\gamma\in\Gamma_K).
  \end{eqnarray}
  Since the function $a$ defined in \eqref{eqn:a(g)} is invariant under scalar multiplication, conjugation, and inverse transpose, we obtain $a(\rho'_{\fq'}(\gamma))=\sigma(a(\rho_\fq(\gamma)))$ for every $\gamma\in\Gamma_K$. This proves Step~(i).

    \underline{Step (ii)}. Prove $\sigma(E)=E'$.

    As in \S\,\ref{Sect: specialization}, there exists an integrally closed, finitely generated $\F_q$-subring $R$ of $K$ containing $B^{\flat}$ such that $\psi^{\flat}$ defines a Drinfeld $B^{\flat}$-module over $R$ for $\flat=\emptyset$ and $'$. We may also assume that the image of the structure morphism $x\mapsto \mathfrak q^{\flat}_x:{\rm Spec}\,R\to {\rm Spec}\,B^{\flat}$ is contained in ${\rm Spec}\,B^{\flat}\setminus\{\mathfrak q^{\flat}\}$.

    Since $\psi$ and $\psi'$ have rank at least $2$, parts~(1) and~(3) of Theorem~\ref{2} imply that $a(\rho^\flat_{\fq^{\flat}}({\rm Frob}_x))\in E^{\flat}$ and that $E^{\flat}$ is generated as a field, by the elements $a(\rho^\flat_{\fq^{\flat}}({\rm Frob}_x))$ for closed points $x\in {\rm Spec}\,R$. Therefore, Step~(ii) follows from Step~(i).

    \underline{Step (iii)}. Prove that there exists an integer $d\in\N$ such that one of the diagrams in \eqref{eqncomdiag} is commutative and $(\sigma|_E)(\mathfrak q)=\fq'$.

    Let $Z$ be the set of closed points $x$ of $ {\rm Spec}\;R$ such that $[\kappa(x):\kappa({\mathfrak q_x})]=[\kappa(x):\kappa({\mathfrak q_x'})]=1$, $\psi$ and $\psi'$ both have ordinary reduction at $x$. By assumption, ${\rm End}_{\overline K}(\psi)=B$ and ${\rm End}_{\overline K}(\psi')=B'$. Then by \cite[Theorem 0.3, Proposition B.8]{P}, $Z$ has Dirichlet density 1 in $ {\rm Spec}\;R$. For any $x\in Z$, write 
    \begin{eqnarray*}
      &&\det(t-\rho_{ \mathfrak q }({\rm Frob}_x))=\sum_{j=0}^mc_{j}(x)t^j\in E_{ \mathfrak q }[t],\\
      &&\det(t-\rho'_{ \mathfrak q' }({\rm Frob}_x))=\sum_{j=0}^mc'_{j}(x)t^j\in E'_{ \mathfrak q' }[t].
    \end{eqnarray*}
    By Theorem \ref{2}, $c^{\flat}_{j}(x)\in B^{\flat}$ for any $j$ and $\mathfrak q^{\flat}_x$ is a principal ideal of $B^{\flat}$ generated by $c^{\flat}_{0}(x)$. We claim that $c_{1}(x)\neq0$.

    Let $\alpha_1,\ldots,\alpha_m$ be the eigenvalues of $\rho_{ \mathfrak q }({\rm Frob}_x)$ in $\overline{F}$. Take an arbitrary valuation $v_{\overline{\mathfrak q}_x}$ of $\overline F$ which extends the normalized valuation of $F$ at $\fq_x$. By \cite[Theorem 3.23]{GL} or \cite[Lemma~2.1]{P}, the Drinfeld $B$-module over $\kappa(x)$ induced by $\psi$ is of height one. Then Theorem~\ref{2}\;(4) tells that the valuations $v_{\overline{\mathfrak q}_x}(\alpha_i)$ are positive for precisely one of $\alpha_i$, and zero for others. Hence $v_{\overline{\mathfrak q}_x}\Big(\sum\limits_{i=1}^m\alpha_1\cdots\widehat{\alpha_i}\cdots\alpha_m\Big)=0$. As a consequence, $c_{1}(x)=(-1)^{m-1}\sum\limits_{j=1}^m\alpha_1\cdots\widehat{\alpha_j}\cdots\alpha_m\neq0.$ This proves $c_{1}(x)\neq0$.

    In the following, we write $C_{E^\flat}$ for the smooth projective curve over $\mathbb{F}_q$ associated with the function field $E^\flat$. In particular, we will identify closed points of $C_{E^\flat}$ with places of $E^{\flat}$. Denote by $({\sigma|_E})^{*}:C_{E'}\simeq C_E$ the isomorphism of projective smooth curves over $\F_p$ induced by ${\sigma|_E}:E\simeq E'$.  Then there is a naturally induced homomorphism of divisor groups $({\sigma|_E})^{*}:{\rm Div}(C_{E'})\to{\rm Div}(C_E)$. Since each square in (\ref{eqncomdiag}) is commutative and $\nu^\flat:E^\flat\to K$ is injective, then to prove this step, it suffices to find $d\in\N$ making one of the two right-lower triangulars in  (\ref{eqncomdiag}) commutative. It remains to show that there exists $d\in\N$ such that one of the following two diagrams commutes:
    \begin{eqnarray}\label{eqncomdiag1}
      \xymatrix{ {\rm Spec}\;R\ar[r]^{{\nu'}^*}\ar[d]^{\nu^*}&C_{E'}\ar[d]^{(\sigma|_E)^*}\\C_E&C_E\ar[l]_{{\rm Fr}^d_{C_E}},}
      \quad\quad\quad
      \xymatrix{ {\rm Spec}\;R\ar[r]^{{\nu'}^*}\ar[d]^{\nu^*}&C_{E'}\ar[d]^{(\sigma|_E)^*}\\C_E\ar[r]^{{\rm Fr}^d_{C_E}}&C_E,}
    \end{eqnarray}
    where ${\nu^\flat}^*$ is the composition ${\rm Spec}\;R\to\Spec\;B^\flat\hookrightarrow C_{E^\flat}$.

    Suppose case (\ref{6}) holds. Applying it to $\gamma={\rm Frob}_x$ for those $x\in Z$, then
    \begin{eqnarray*}\label{7}
      c'_{j}(x)=\chi({\rm Frob}_x)^{m-j}\frac{\sigma(c_{m-j}(x))}{\sigma(c_{0}(x))}\in E'_{\fq'}\quad\quad(0\leq j\leq m).
    \end{eqnarray*}
    In particular, taking $j=m-1$ and $j=0$ respectively, as $c_{1}(x)\neq0$, we have
    \begin{eqnarray*}
      \chi({\rm Frob}_x)=c'_{m-1}(x)\cdot\frac{\sigma(c_{0}(x))}{\sigma(c_{1}(x))}\in E'
    \end{eqnarray*}
    and
    \begin{eqnarray}\label{Eqn: chi(Frob_x^m)}
      \chi({\rm Frob}_x)^m=c'_{0}(x)\cdot\sigma(c_{0}(x))\in E'.
    \end{eqnarray}
    Consider both sides of \eqref{Eqn: chi(Frob_x^m)} as functions of the curve $C_{E'}$, and taking the corresponding principal divisors, we obtain
    \begin{eqnarray}\label{Eqn: Frob_x^m}
      {\fq}_x'-\frac{\deg(\fq'_x)}{\deg(\infty_{E'})}\infty_{{E'}}+{({\sigma|_E})^{*}}^{-1}(\mathfrak q_x)-\frac{\deg(\fq_x)}{\deg(\infty_{E})}{({\sigma|_E})^{*}}^{-1}(\infty_E)\in m{\rm Div}(C_{E'}).
    \end{eqnarray}
    Here $\mathfrak q_x'$ (resp.\ $(\sigma|_E)^{*^{-1}}(\mathfrak q_x)$) is the divisor of zeros of $c'_0(x)$ (resp.\ $\sigma(c_0(x))$) by Theorem~\ref{2} (2), and it appears with multiplicity one since $x \in Z$. Moreover, by \cite[Proposition 4.7.17]{Go}, $\infty_{E^{\flat}}$ is the unique closed point of $C_{E^{\flat}}$ above $\infty\in C$. After possibly shrinking $\Spec R$, we may further assume that $(\sigma|_E)^{*^{-1}}(\infty_E)$ is not contained in the image of $\Spec R \to \Spec B'$. Under this assumption, $\mathfrak q_x'$ is distinct from both $(\sigma|_E)^{*^{-1}}(\infty_E)$ and $\infty_{E'}$. Depending on whether $(\sigma|_E)^{*^{-1}}(\mathfrak q_x)$ coincides with $\mathfrak q_x'$ or not, the coefficient of $\mathfrak q_x'$ in \eqref{Eqn: Frob_x^m} is equal to $2$ or $1$, respectively. In either case, this coefficient is not divisible by $m \ge 3$, yielding a contradiction. Hence only case~\eqref{9} can occur.

    Repeating the above argument, we obtain
    \[
      c'_j(x)
      = \chi(\mathrm{Frob}_x)^{m-j}\sigma(c_j(x))
      \qquad (0 \le j \le m).
    \]
    In particular, taking $j=0$ and $j=1$, we have
    \[
      \chi(\mathrm{Frob}_x)^{m-1}
      = \frac{c'_1(x)}{\sigma(c_1(x))} \in E',
      \qquad
      \chi(\mathrm{Frob}_x)^m
      = \frac{c'_0(x)}{\sigma(c_0(x))} \in E'.
    \]
    It follows that $\chi(\mathrm{Frob}_x) \in E'$, and hence
    \[
      {\fq}_x'-\frac{\deg(\fq'_x)}{\deg(\infty_{E'})}\infty_{{E'}}-{({\sigma|_E})^{*}}^{-1}(\mathfrak q_x)+\frac{\deg(\fq_x)}{\deg(\infty_{E})}{({\sigma|_E})^{*}}^{-1}(\infty_E)\in m{\rm Div}(C_{E'}).
    \]

    Since $\mathfrak q_x \ne (\sigma|_E)^*(\infty_{E'})$ and $\mathfrak q_x \ne \infty_E$, and $m \ge 2$, it follows that $(\sigma|_E)^*(\mathfrak q_x') = \mathfrak q_x$ for all $x \in Z$. Consequently, the two morphisms
    \[
      \Spec R \xrightarrow{\nu^*} C_E
      \quad \text{and} \quad
      \Spec R \xrightarrow{\nu'^*} C_{E'}
      \xrightarrow{(\sigma|_E)^*} C_E
    \]
    agree on $Z$ set-theoretically. By Theorem~\ref{thmf=g}, there exists $d \in \mathbb{N}$ such that one of the diagrams in~(\ref{eqncomdiag1}) commutes. This completes the proof of Step~(iii). In particular, due to the uniqueness of $\infty_{E}$ and $\infty_{E'}$, we know that $\sigma(\infty_E)=\infty_{E'}$, which implies that $\sigma(B)\subset B'$. 

    \underline{Step (iv).} Prove that $\phi$ and $\phi'$ are isogenous up to a Frobenius twist.

    Without loss of generality, we may choose $d \in \mathbb{N}$ such that the first diagram in~(\ref{eqncomdiag}) is commutative.
    Denote by $\widetilde{f}$ the composition of $E_{\mathfrak q}$-linear isomorphisms
    \[
      V_{\mathfrak q}(\psi^{(d)})
      \overset{\mathrm{Lemma}~\ref{lemfrobtwist}}{\simeq}
      (V_{\mathfrak q}\psi)^{(d)}
      \xrightarrow{f^{(d)}}
      (V_{\mathfrak q'}\psi')^{(d)}
      \overset{\mathrm{Lemma}~\ref{lemfrobtwist}}{\simeq}
      V_{\mathfrak q'}(\psi' \circ \mathrm{Fr}^d_{B'})
      \simeq
      V_{\mathfrak q}(\psi' \circ \mathrm{Fr}^d_{B'} \circ \sigma|_B)
      =
      V_{\mathfrak q}(\psi' \circ \sigma|_B \circ \mathrm{Fr}^d_B),
    \]
    where $\psi^{(d)}$, $(V_{\mathfrak q}\psi)^{(d)}$, and $(V_{\mathfrak q'}\psi')^{(d)}$ denote the Frobenius twists of $\psi$, $V_{\mathfrak q}\psi$, and $V_{\mathfrak q'}\psi'$, respectively (see Definition~\ref{defnfrobtwist}), and $f^{(d)}$ is the Frobenius twist of the map $f$ in~\eqref{Eqn: sigma_semilinear_f}.

    After possibly replacing $K$ by a finite extension, we may assume that $\psi^{(d)}$ is defined over $K$. Then the two Drinfeld $B$-modules $\psi^{(d)}$ and $\psi' \circ \sigma|_B \circ \mathrm{Fr}^d_B$ have the same characteristic $\nu|_B : B \to K$ by~(\ref{eqncomdiag}). Denote by $\theta_{\mathfrak q}$ and $\theta'_{\mathfrak q}$ their associated $\mathfrak q$-adic Galois representations. Then~(\ref{9}) can be rewritten as
    \begin{equation}\label{eqntheta}
      \theta'_{\mathfrak q}(\gamma)
      =
      \sigma^{-1}(\chi(\gamma)) \cdot
      \bigl(\widetilde{f} \circ \theta_{\mathfrak q}(\gamma) \circ \widetilde{f}^{-1}\bigr)
      \qquad (\gamma \in \Gamma_K).
    \end{equation}

    By an argument analogous to Simplification~3, we may assume that the determinant Drinfeld modules associated to $\theta_{\mathfrak q}$ and $\theta'_{\mathfrak q}$ are isogenous over $K$. Taking determinants in~\eqref{eqntheta}, we obtain $\chi(\gamma)^m = 1$. After replacing $K$ by a further finite extension if necessary, we may assume that the character $\chi$ is trivial. Consequently,~\eqref{eqntheta} yields a $\Gamma_K$-equivariant $E_{\mathfrak q}$-linear isomorphism
    $
    V_{\mathfrak q}(\psi^{(d)})
    \simeq
    V_{\mathfrak q}(\psi' \circ \sigma|_B \circ \mathrm{Fr}^d_B).
    $
    By \cite[Theorem~1.7]{P}, it follows that the Drinfeld $B$-modules $\psi^{(d)}$ and $\psi' \circ \sigma|_B \circ \mathrm{Fr}^d_B$ are isogenous over $K$. Restricting to $A$, we conclude that $\phi^{(d)}$ and $\phi'$ are isogenous. This completes the proof.
  \end{proof}

  \begin{rem}
    In \underline{Step (iv)}, we deduce the desired result from~\eqref{eqntheta} rather than from~\eqref{9}, because in~\eqref{eqntheta} the map $\widetilde{f}$ is $E_{\mathfrak q}$-linear, whereas in~\eqref{9} the map $f$ is only semilinear.
  \end{rem}

Denote by $G^{\rm der}$ the closure of the commutator subgroup of a profinite group $G$. Thus $\Gamma_K^{\rm der}={\rm Gal}(\overline K/K^{\rm ab})$ where $K^{\rm ab}$ is the maximal abelian extension of $K$ in $\overline K$. By Definition \ref{defngamma_S}, $$(\bar\rho_\fq,\bar\rho_{\fq'})(\Gamma_K)^\sim =(\rho_\fq,\rho'_{\fq'})(\Gamma_K^{\rm der}).$$ The result of this section can be summarized as follows.

\begin{thm}\label{thmmainsl}
Suppose that $K$ is a finitely generated extension of $F$, and
that $\psi$ and $\psi'$ both have rank at least two.
Then the following statements are equivalent:
\begin{enumerate}[label=(\arabic*)]
  \item $\phi$ and $\phi'$ are not geometrically isogenous up to
  any Frobenius twist.

  \item For every pair
  $(\fq,\fq')\in{\rm Max}\,B\times{\rm Max}\,B'$, the image
  $
    (\rho_\fq,\rho'_{\fq'})(\Gamma_K^{\rm der})
  $
  is open in
  $
    {\rm SL}_{B_\fq}(T_\fq\psi)
    \times
    {\rm SL}_{B'_{\fq'}}(T_{\fq'}\psi').
  $

  \item There exists $\fp\in{\rm Max}\,A$ such that, for every
  $\fq\mid\fp$ and $\fq'\mid\fp$, the image
  $
    (\rho_\fq,\rho'_{\fq'})(\Gamma_K^{\rm der})
  $
  is open in
  $
    {\rm SL}_{B_\fq}(T_\fq\psi)
    \times
    {\rm SL}_{B'_{\fq'}}(T_{\fq'}\psi').
  $
\end{enumerate}
\end{thm}

\begin{proof}
The implication $(1)\Rightarrow(2)$ follows from
Lemma~\ref{3}, Theorem~\ref{11}, and
Theorem~\ref{thmminimalmodel}.
The implication $(2)\Rightarrow(3)$ is immediate.

To prove $(3)\Rightarrow(1)$, suppose, to the contrary, that $\phi$ and $\phi'$ are geometrically isogenous up to a Frobenius twist. After interchanging the two modules and passing to a finite extension of $K$, we may assume that $\phi'$ is isogenous to $\phi^{(d)}$ for some $d\geq0$.

By Lemma~\ref{lem:frobtwist-endomorphisms},  the Frobenius twist and the isogeny induce an isomorphism $\sigma:E\xrightarrow{\sim}E'$ identifying $B$ with $B'$. For a given $\fp\in {\rm Max}\,A$, and take $\fq\in {\rm Max}\,B$ lying above $\fp$. Then $\fq'=\sigma(\fq)$ also lies above $\fp$. Moreover, at these corresponding primes, there is an induced $\widehat{\sigma}$-semilinear isomorphism between the projective Galois representations. Hence
$
  (\bar\rho_\fq,\bar\rho'_{\fq'})(\Gamma_K)
$
is contained in the graph of an isomorphism between the corresponding projective linear groups. Since $\psi$ and $\psi'$ have rank at least two, it follows that
$
  (\rho_\fq,\rho'_{\fq'})(\Gamma_K^{\rm der})
$
is not open in
$
  {\rm SL}_{B_\fq}(T_\fq\psi)
  \times
  {\rm SL}_{B'_{\fq'}}(T_{\fq'}\psi').
$
This gives a contradiction. Therefore $(1)$ holds.
\end{proof}


  \section{Determinant representation of \texorpdfstring{$\psi$ and $\psi'$}{psi and psi'}}\label{Sect: determinant}

  In \S\,\ref{Sect: adj_repn}, we analyzed the adjoint representations attached to $\psi$ and $\psi'$. In this section, we study the determinant representation of $(\rho,\rho')$ using explicit class field theory for rank one Drinfeld modules. In \S\,\ref{Sect: open_norm_map},  we establish an openness result for norm maps between idele groups of global fields. In \S\,\ref{Sect: Gal_repn_rank_1}, we recall the Galois representation attached to a rank-one Drinfeld module from the viewpoint of class field theory.  Finally, in \S\,\ref{Sect: det_repn_open}, we apply these results to investigate the determinant representation $\det(\rho),\det(\rho')$ associated with $\psi$ and $\psi'$ respectively.


  \subsection{Openness of the norm map of global fields}\label{Sect: open_norm_map}
  In this subsection, we study the openness of the norm map between idele groups of global fields. We begin with a commutative diagram of fields
  \begin{equation}\label{Diag: L}
    \xymatrix{L\ar[r]\ar[d]&L_1\ar[d]\\L_2\ar[r]&L_3}
  \end{equation}
  such that $[L_3:L]<\infty$. Consider the following complex of algebraic groups over $L$:
  \begin{equation}\label{A1}
    {\rm Res}_{L_3/L}\G_{{\rm m},L_3}\x{({\rm N}_{L_3/L_1},{\rm N}_{L_3/L_2})}{\rm Res}_{L_1/L}\G_{{\rm m},L_1}\times{\rm Res}_{L_2/L}\G_{{\rm m},L_2}\x{{\rm N}_{L_1/L}\cdot{\rm N}_{L_2/L}^{-1}}\G_{{\rm m},L}\to 1.
  \end{equation}
  Let $Q$ denote the kernel of ${\rm N}_{L_1/L}\cdot{\rm N}_{L_2/L}^{-1}$. Then the map $({\rm N}_{L_3/L_1},{\rm N}_{L_3/L_2})$ naturally factors through
  \begin{equation}\label{A11}
    \pi:{\rm Res}_{L_3/L}\G_{{\rm m},L_3}\to Q.
  \end{equation}

  \begin{lem}\label{thmA1}
    Suppose $L_3/L$ is separable and $L_1\cap L_2=L$. Then (\ref{A1}) is an exact sequence of tori over $L$, and $\pi$ is a smooth morphism with geometrically connected fibers.
  \end{lem}

  \begin{proof}
    Since $L_3/L$ is separable, every term of (\ref{A1}) is a torus. By applying the functor ${\rm Hom}(\bullet\times_L\overline L,\G_{{\rm m},\overline L})$ to (\ref{A1}), it suffices to show that the induced complex of character groups
    \begin{eqnarray}\label{eqnA2}
      0\to\Z\x{\alpha}\Z^{{\rm Hom}_L(L_1,\overline L)}\oplus \Z^{{\rm Hom}_L(L_2,\overline L)} \x{\beta}\Z^{{\rm Hom}_L(L_3,\overline L)}
    \end{eqnarray}
    is exact. Here $\Z^{{\rm Hom}_L(L_i,\overline L)}$ denotes the set of maps from ${\rm Hom}_L(L_i,\overline L)$ to $\Z$; $\alpha$ is the diagonal map; and for
    $
    f_i\in \Z^{\Hom_L(L_i,\overline L)} \quad (i=1,2),
    $
    the map $\beta(f_1,f_2)$ is defined by
    \[
      \beta(f_1,f_2)(\varphi)
      =
      f_1(\varphi|_{L_1})-f_2(\varphi|_{L_2})
      \qquad
      (\varphi\in \Hom_L(L_3,\overline L)).
    \]

    We first prove the exactness of \eqref{eqnA2}. Note that
    $$
    \ker\beta=\bigl\{
      (f_1,f_2)\in \Z^{\Hom_L(L_1,\overline L)}\oplus\Z^{\Hom_L(L_2,\overline L)} \;\big|\;
      f_1(\varphi|_{L_1})=f_2(\varphi|_{L_2})
      \text{ for all }\varphi\in \Hom_L(L_3,\overline L)
    \bigr\}.$$
    Let $\sim$ be the finest equivalence relation on $\Hom_L(L_1,\overline L)\sqcup \Hom_L(L_2,\overline L)$ such that
    \begin{center}
      \emph{$\varphi|_{L_1}\sim \varphi|_{L_2}
        \qquad
      \text{for every }\varphi\in \Hom_L(L_3,\overline L)$.}
    \end{center}
    Therefore, to prove exactness of \eqref{eqnA2}, it remains to show that $(\Hom_L(L_1,\overline L)\sqcup \Hom_L(L_2,\overline L))/\sim$ consists of a single equivalence class.

    After replacing $L_3$ by a finite Galois extension of $L$ containing it, we may assume that $L_3/L$ is Galois with Galois group $G$. This does not affect the claim. Let $H_i={\rm Gal}(L_3/L_i)$ for $i=1,2$. Fix $\varphi_0\in{\rm Hom}_L(L_3,\overline L)$. The bijection $g\mapsto \varphi_0\circ g$ on $G\simeq{\rm Hom}_L(L_3,\overline L)$ induces a commutative diagram of sets
    \[
      \xymatrix{G\ar[r]^\simeq\ar[d]&{\rm Hom}_L(L_3,\overline L)\ar[d]\\G/H_i\ar[r]^\simeq&{\rm Hom}_L(L_i,\overline L).}
    \]
    Under these identifications, the relation $\sim$ becomes the \emph{finest equivalence relation} on $G/H_1 \sqcup G/H_2$ such that
    \[
      gH_1\sim gH_2
      \qquad
      \text{\emph{for all} }g\in G.
    \]
    In other words, for any $g,g'\in G$, if $g^{-1}g'\in H_1H_2$, we have $gH_1\sim g'H_2\sim g'H_1$.

    We now show that $(G/H_1\sqcup G/H_2)/\sim$ has only one element. Since $L_1\cap L_2=L$, Galois theory implies that the subgroups $H_1$ and $H_2$ generate $G$. Hence for any $g,g'\in G$, we may write
    \[
      g^{-1}g' = u_1u_2\cdots u_r
    \]
    for some $u_i\in H_1H_2$. There exist $g_0,\ldots,g_r\in G$ such that $g_0=g$, $g_r=g'$ and $g_i^{-1}g_{i+1}=u_{i+1}$ for any $0\leq i\leq r-1$. Hence
    $g_iH_1\sim g_{i+1}H_1.$
    By transitivity, we obtain $gH_1\sim g'H_1\sim g'H_2$ for all $g,g'\in G$.
    Therefore $(G/H_1\sqcup G/H_2)/\sim$ has exactly one element.
    This proves that \eqref{eqnA2} is exact.

    It remains to prove that $\pi$ is smooth with geometrically connected fibers. For this, it is enough to show that ${\rm coker}(\beta)$ is torsion free. This is equivalent to verifying that if $df_3\in{\rm im}(\beta)$ for some nonzero integer $d$, then  $f_3\in{\rm im}(\beta)$ too. Now suppose $d f_3 =\beta(f_1,f_2)$ for $f_i\in \Z^{{\rm Hom}_L(L_i,\overline L)}$ $(i=1,2)$.
    This means that $f_1(\varphi|_{L_1})\equiv f_2(\varphi|_{L_2})\pmod d$ for any $\varphi\in{\rm Hom}_L(L_3,\overline L)$. Since $(\Hom_L(L_1,\overline L)\sqcup \Hom_L(L_2,\overline L))/\sim$ is a singleton‌, there exists $a\in\Z$ such that
    $$f_1(\varphi_1)\equiv f_2(\varphi_2)\equiv a\pmod d\hbox{ for any }\varphi_i\in{\rm Hom}_L(L_i,\overline L), \;i=1,2.$$
    Hence $\frac1d(f_i-a)\in\Z^{{\rm Hom}_L(L_i,\overline L)}$ and hence $f_3=\beta\big(\frac1d(f_1-a),\frac1d(f_2-a)\big)\in{\rm im}(\beta)$.
    This completes the proof.
  \end{proof}

  Suppose that $L$ is a global field.  For $i=\emptyset,1,2,3$, denote by $\A_{L_i}$ and $I_{L_i}$ the adele rings and idele groups of $L_i$, respectively.
  For any embedding from $L_j$ to $L_i$ occurring in \eqref{Diag: L}, denote by ${\rm N}_{L_i/L_j}:I_{L_i}\to I_{L_j}$ the corresponding norm map.

  \begin{prop}\label{thmopenness}
    Suppose $L$ is a global field and $L_1\cap L_2=L$. Then the morphism $\pi$ defined in (\ref{A11}) induces an open map $I_{L_3}\to Q(\A_L)$. In particular,
    $\{({\rm N}_{L_{3}/L_{1}}(x),{\rm N}_{L_{3}/L_{2}}(x) )\mid x\in I_{L_3} \}$ is an open subgroup of
    $$Q(\A_L)=\{(x_1,x_2)\in I_{L_1}\times I_{L_2}\mid{\rm N}_{L_{1}/L}(x_1)={\rm N}_{L_{2}/L}(x_2)\}.$$
  \end{prop}

  \begin{proof}
    For each $i=1,2,3$, we write $L_i'$ for the separable closure of $L$ inside $L_i$.   
    Taking the $\A_L$-points of (\ref{A1}) yields a commutative diagram
    \begin{equation}\label{Diag: ideles}
      \xymatrix{I_{L_3}\ar[rrr]^{({\rm N}_{L_{3}/L_{1}},{\rm N}_{L_{3}/L_{2}} )  }\ar[dd]^{{\rm N}_{L_{3}/L'_{3}}}&&&I_{L_1}\times I_{L_2}\ar[rrr]^{{\rm N}_{L_{1}/L_{}}\cdot{\rm N}_{L_{2}/L_{}}^{-1}}\ar[dd]^{{\rm N}_{L_{1}/L'_{1}}\times{\rm N}_{L_{2}/L'_{2}}}&&&I_L\ar[dd]\\\\
        I_{L_3'}\ar[rrr]^{({\rm N}_{L'_{3}/L'_{1}},{\rm N}_{L'_{3}/L'_{2}} )  }&&&I_{L_1'}\times I_{L_2'}\ar[rrr]^{{\rm N}_{L'_{1}/L}\cdot{\rm N}_{L'_{2}/L}^{-1}}&&&I_L
      }
    \end{equation}
    of topological groups. By the definition of $L_i'$, we know that $L_i/L_i'$ is purely inseparable. By \cite[Corollary to Proposition 7.4]{Ro}, each ${\rm N}_{L_i/L_i'}$ is bijective. This reduces the theorem to the case when $L_3/L$ is separable. 
    In this case, $\pi$ is a smooth surjective morphism of schemes with geometrically connected fibers  by Lemma~\ref{thmA1}. Applying \cite[Theorem 4.5]{Co}, we conclude that the induced map $({\rm Res}_{L_3/L}\G_{{\rm m},L_3})(\A_L)\to Q(\A_L)$ is open. Then the theorem follows immediately from the fact that $({\rm Res}_{L_3/L}\G_{{\rm m},L_3})(\A_L)=\G_{{\rm m},L_3}(L_3\otimes_L\A_L)=I_{L_3}$.
  \end{proof}

  More generally, for a set $\Sigma$ of places of $L$ and for $i\in \{\emptyset, 1,2,3\}$, let $v$ be a place of $L_i$ lying above an element in $\Sigma$. We write $L_{i,v}$ to be the corresponding local field. Define $I_{L_i}^\Sigma$ to be the restricted product of $L_{i,v}^\times$ where $v$ ranges over all places of $L_i$ lying above $\Sigma$. In this case we also have the norm maps  ${\rm N}_{L_i/L_j}:I_{L_i}^\Sigma\to I_{L_j}^\Sigma$ according to the diagram \eqref{Diag: L}.
  \begin{rem}
    Suppose $L$ is a global field and $L_1\cap L_2=L$. Then for any set $\Sigma$ of places of $L$,
    $\{({\rm N}_{L_{3}/L_{1}}(x),{\rm N}_{L_{3}/L_{2}}(x) )\mid x\in I_{L_3}^\Sigma \}$ is open in $\{(x_1,x_2)\in I_{L_1}^\Sigma\times I_{L_2}^\Sigma\mid{\rm N}_{L_{1}/L}(x_1)={\rm N}_{L_{2}/L}(x_2)\}.$
  \end{rem}

  By the same arguments as in Proposition~\ref{thmopenness}, we have the following corollary.
  \begin{cor}\label{coropenness}
    For any finite extension $L_1/L$ of global fields, the norm map ${\rm N}_{L_1/L}:I_{L_1}\to I_L$ is open.
  \end{cor}

  \subsection{Galois representations of rank-one Drinfeld modules}\label{Sect: Gal_repn_rank_1}
  In this subsection, we study the adelic representation attached to Drinfeld modules of rank one via global class field theory.

  Let
  \[
    I_F^{\mathrm f}
    =
    \left\{
      (a_{\mathfrak p})\in \prod_{\mathfrak p\in {\rm Max}\,A} F_{\mathfrak p}^\times
      \;\middle|\;
      a_{\mathfrak p}\in A_{\mathfrak p}^\times \text{ for almost all } \mathfrak p
    \right\}
  \]
  be the group of finite ideles of $F$. There is a natural surjective homomorphism $a\mapsto a^{\rm f}:I_F\twoheadrightarrow I_F^{\rm f}$.

  Let
  $
  \widehat A=\varprojlim\limits_{I} A/I,
  $
  when $I$ runs through all nonzero ideals of $A$. Suppose that $\phi$ is a Drinfeld $A$-module of rank one over $K$. Then the natural action of $\Gamma_K$ on the torsion submodule
  $
  \phi(K^{\mathrm{sep}})_{\mathrm{tors}}
  $
  of the $A$-module $\phi(K^{\mathrm{sep}})$ gives rise to a Galois representation
  $$
  \rho:{\rm Gal}(K^{\rm sep}/K)\to{\rm Aut}_A(\phi(K^{\rm sep})_{\rm tors})\simeq\widehat A^\times
  $$
  which factors through $${\rho}^{\rm ab}:{\rm Gal}(K^{\rm ab}/K)\to \widehat A^\times.$$
  On the other hand, if $K$ is a function field, global class field theory provides the reciprocity homomorphism
  \[
    {\rm rec}_K: I_K\to {\rm Gal}(K^{\rm ab}/K)
  \]
  which sends an idele $a\in I_K$ to its Artin symbol $(a,K^{\rm ab}/K)$ in ${\rm Gal}(K^{\mathrm{ab}}/K)$.

  \begin{prop}\label{thmrankoneDrinfeld}
    Suppose $\phi$ is of rank one and $[K:F]<\infty$.
    \begin{itemize}
      \item [(1)] There exists a continuous homomorphism $\epsilon:I_K\to F^\times$ (here $F^\times$ has discrete topology) such that for any $a\in I_K$, we have
        $$\rho^{\rm ab}((a,K^{\rm ab}/K))=\epsilon(a)\cdot{\rm N}_{K/F}(a^{-1})^{\rm f}\in I_F^{\rm f}.$$
      \item[(2)] $\rho(\Gamma_K)$ is open in $ \widehat A^\times$.
    \end{itemize}
  \end{prop}
  \begin{proof}
Consider the continuous map 
 \begin{eqnarray*}
\epsilon:I_K&\to& I_F^{\rm f}\\
a&\mapsto&\rho^{\rm ab}((a,K^{\rm ab}/K))\cdot{\rm N}_{K/F}(a)^{\rm f}.
 \end{eqnarray*}
to prove (1), it suffices to show $\epsilon(a)$ lies in the diagonal copy of $F^\times$.

    Let $F^{{\rm ab},\infty}$ be the maximal abelian extension of $F$ in $\overline F$ in which $\infty$ splits completely, and let $(\phi^{\rm univ},\iota^{\rm univ})$ be the universal rank one Drinfeld module over $F^{{\rm ab},\infty}$ with a full level structure. Fix a full level structure $\iota:F/A\to\phi(K^{\rm ab})$ of $\phi$ over $K^{\rm ab}$. Then there exists an $F$-embedding $f:F^{{\rm ab},\infty}\hookrightarrow K^{\rm ab}$ such that $f^*(\phi^{\rm univ},\iota^{\rm univ})$ is isomorphic to $(\phi,\iota)$.

    Choose $b\in\widehat A$ and $u\in F^\times$ such that ${\rm N}_{K/F}(a)^{\rm f}=bu$. Let $\mathfrak b=b\widehat A\cap A$. Recalling the action of ideals on Drinfeld modules \cite[Proposition 4.9.3]{Go}, we have the isogeny $\phi_\mathfrak b:\phi\to\mathfrak b*\phi$. By \cite[Theorem 7.1.6]{Go}, the action of $I_F^{\rm f}$ on Drinfeld modules with full level structures (hence on $F^{{\rm ab},\infty}$) agrees with that given by class field theory. By \cite[p.~48]{Le}, the action $({\rm N}_{K/F}(a)^{\rm f})_*(\phi,\iota)$ of ${\rm N}_{K/F}(a)^{\rm f}$ on $(\phi,\iota)$ is isomorphic to $(\mathfrak b*\phi,\iota')$ such that the following diagram commutes:
    \[\xymatrix{F/A\ar[rr]^\iota\ar[d]^b&&\phi(K^{\rm ab})\ar[d]^{\phi_\mathfrak b}\\F/A\ar[rr]^{\iota'}&&(\mathfrak b*\phi)(K^{\rm ab}).}\]
    Since $F^{\rm ab,\infty}/F$ is totally split at $\infty$, the action of $I_F$ on $F^{\rm ab,\infty}$ factors through $I_F^{\rm f}$. By class field theory, we have the commutative diagram
    \[\xymatrix{F^{{\rm ab},\infty}\ar[rrr]^{({\rm N}_{K/F}(a)^{\rm f},F^{{\rm ab},\infty}/F)}\ar[d]^f&&&F^{{\rm ab},\infty}\ar[d]^f\\K^{\rm ab}\ar[rrr]^{(a,K^{\rm ab}/K)}&&&K^{\rm ab}.}\]
    Hence
    \begin{eqnarray}\label{eqncft}(\mathfrak b*\phi,\iota')\simeq(a,K^{\rm ab}/K)^*(\phi,\iota)\simeq(\phi,(a,K^{\rm ab}/K)\circ\iota).
    \end{eqnarray}
    In particular, $\phi\simeq\mathfrak b*\phi$. 
    
    Thus by \cite[Corollary 4.9.5]{Go}, the ideal $\mathfrak b$ of $A$ is generated by an element $c\in A$. Hence $b\widehat A\cap A=Ac$ implies $\frac bc\in\widehat A^\times$. We have ${\rm N}_{K/F}(a)^{\rm f}=bu=(\frac bc)(cu)$. Replacing $b$ by $\frac bc$ and $u$ by $cu$, we may assume $b\in\widehat A^\times$. By \cite[5.4]{Le} and (\ref{eqncft}), we have
    $$(\phi,\iota\circ b^{-1})\simeq b_*(\phi,\iota)\simeq(\mathfrak b*\phi,\iota')\simeq(\phi,(a,K^{\rm ab}/K)\circ\iota).$$
    Since ${\rm Aut}(\phi)\simeq\F_q^\times$, there exists $\alpha\in \F_q^\times$ such that $\iota\circ b^{-1}\alpha=(a,K^{\rm ab}/K)\circ\iota:F/A\to\phi(K^{\rm ab})$. In other words, the following diagram commutes:
    \[\xymatrix{F/A\ar[d]^{b^{-1}\alpha}\ar[rr]^{\iota}&&\phi(K^{\rm ab})\ar[d]^{(a,K^{\rm ab}/K)}\\F/A\ar[rr]^\iota&&\phi(K^{\rm ab}).}\]
    Since $\iota$ induces a bijection $F/A\simeq\phi(K^{\rm sep})_{\rm tors}$, we have $\rho^{\rm ab}((a,K^{\rm ab}/K))=b^{-1}\alpha\in\widehat A^\times$. So $\epsilon(a)=\rho^{\rm ab}((a,K^{\rm ab}/K)){\rm N}_{K/F}(a)^{\rm f}=(b^{-1}\alpha)(bu)=\alpha  u\in F^\times$. This proves (1). 
    
   We have $\rho^{\rm ab}((a,K^{\rm ab}/K))={\rm N}_{K/F}(a^{-1})^{\rm f}$ for any $a$ in $\ker(\epsilon)$, which is open in $I_K$ by (1). By Corollary \ref{coropenness},  the norm map ${\rm N}_{K/F}:I_K\to I_F$ is open. This proves the openness of $\rho$.
    

  \end{proof}

  \subsection{Determinant representation associated to \texorpdfstring{$\psi$ and $\psi'$}{psi and psi'}}\label{Sect: det_repn_open}
  The Galois representation \eqref{eqnvarrho} associated with $\phi$ and $\phi'$ factors as
\begin{equation}\label{eqnvarrhop2}
 \Gamma_K \to \prod_{\mathfrak q\in\mathrm{Max}\,B} \mathrm{GL}_{B_\mathfrak q}(T_\mathfrak q\psi) \times \prod_{\mathfrak q'\in\mathrm{Max}\,B'} \mathrm{GL}_{B'_{\mathfrak q'}}(T_{\mathfrak q'}\psi').
\end{equation}
We still denote it by $(\rho,\rho')$ by abuse of notation.
  We are ready to study the determinant representation $(\det(\rho),\det(\rho')):\Gamma_K\to I_E^{\rm f}\times I_{E'}^{\rm f}$ of the representation (\ref{eqnvarrhop2}) associated with $\psi$ and $\psi'$, where $I_{E^\flat}^{\rm f}$ is the group of finite ideles of $E^\flat$ relative to the unique place $\infty_{E^\flat}$ of $E^\flat$ above $\infty$. Define $I^{\rm f}_{F'}$ for $F'=E\cap E'$ likewise.

  \begin{thm}\label{Thm: det_open}
 Suppose $K$ is a finitely generated field extension of $F$. Then $(\det(\rho),\det(\rho'))(\Gamma_K)$ contains an open subgroup of $H\coloneqq\{(g,g')\in I_E^{\rm f}\times I_{E'}^{\rm f}\mid{\rm N}_{E/F'}(g)={\rm N}_{E'/F'}(g')\in I_{F'}^{\rm f}\}$.
  \end{thm}
  \begin{proof}
    The determinant representation factors through
    $$(\det(\rho)^{\rm ab},\det(\rho')^{\rm ab}):{\rm Gal}(K^{\rm ab}/K)\to I_E^{\rm f}\times I_{E'}^{\rm f}.$$
    Note that the determinant Drinfeld module $\det(\psi^{\flat})$ of $\psi^{\flat}$ is isomorphic to a Drinfeld $B^{\flat}$-module over some finite extension of $E^{\flat}$ in $\overline K$. Hence, after replacing $K$ with a finite extension of itself, we may assume these determinant Drinfeld modules are already defined over $K$ and $[K:F]<\infty$. By Proposition~\ref{thmrankoneDrinfeld}, there exists a continuous homomorphism $\epsilon^{\flat}:I_K\to (E^{\flat})^\times$ such that
    $$(\det(\rho)^{\rm ab},\det(\rho')^{\rm ab})((a,K^{\rm ab}/K))=\epsilon^\flat(a)\cdot{\rm N}_{K/E^{\flat}}(a^{-1})^{\rm f}\in I_{E^{\flat}}^{\rm f}\quad(a\in I_K).$$
    So the group $\Omega\coloneqq\ker(\epsilon)\cap\ker(\epsilon')$ is open in $I_K$. For any $a\in\Omega$, we thus have
    \[
      (\det(\rho)^{\rm ab},\det(\rho')^{\rm ab})((a,K^{\rm ab}/K))=({\rm N}_{K/E}(a^{-1})^{\rm f},{\rm N}_{K/E'}(a^{-1})^{\rm f})\in H.
    \]
    By Proposition~\ref{thmopenness}, $\{(\det(\rho)^{\rm ab},\det(\rho')^{\rm ab})(a)\mid a\in\Omega\}$ is open in $H$. This completes the proof.
  \end{proof}

  \section{Image of the residue representation}\label{Sect: img_res_repn}

  In Theorem \ref{thmmainsl}, we have shown that if $\phi$ and $\phi'$ are not geometrically isogenous up to any Frobenius twist, then the image $(\rho_\fq,\rho'_{\fq'})(\Gamma_K^{\rm der})$ is open in ${\rm SL}_{B_{\fq}}(T_{\fq}\psi)\times{\rm SL}_{B'_{\fq'}}(T_{\fq'}\psi')$ for any pair $(\fq,\fq')$. To pass from this local openness to a simultaneous commensurability statement for infinitely many primes, we must study the residual Galois representations. For each maximal ideal $\fq^\flat$ of $B^\flat$, let $\kappa(\fq^\flat)=B^\flat/\fq^\flat$ be its residue field. We denote by 
  $
\widetilde\rho^\flat_{\fq^\flat}:\Gamma_K\to{\rm GL}_{\kappa(\fq^\flat)}(\psi^\flat[\fq^\flat])
$
the natural residue representation associated to $\rho^\flat_{\fq^\flat}$. For any pair $(\fq,\fq')\in{\rm Max}\,B\times{\rm Max}\,B'$, we consider the joint residual representation
\[
(\widetilde\rho_\fq,\widetilde\rho'_{\fq'}):\Gamma_K\to{\rm GL}_{\kappa(\fq)}(\psi[\fq])\times{\rm GL}_{\kappa(\fq')}(\psi'[\fq']).
\]
The main result of this section is Proposition~\ref{prop:residuerepren}, which shows that if $\phi$ and $\phi'$ are not geometrically isogenous up to any Frobenius twist, then there exist finite subsets $I\subset{\rm Max}\,B$ and $I'\subset{\rm Max}\,B'$ such that
\[
(\widetilde\rho_\fq,\widetilde\rho'_{\fq'})(\Gamma_K^{\rm der})={\rm SL}_{\kappa(\fq)}(\psi[\fq])\times{\rm SL}_{\kappa(\fq')}(\psi'[\fq'])
\]
for every pair $(\fq,\fq')\in({\rm Max}\,B\setminus I)\times({\rm Max}\,B'\setminus I')$.

To state the main result in this section, we need the following preliminaries.

\begin{prop}\label{prop:residual}
Assume that $\psi$ and $\psi'$ both have rank
at least two. Given a pair $(\fq,\fq')\in{\rm Max}\,B\times{\rm Max}\,B'$ satisfies the following three conditions:
\begin{enumerate}[label=(\alph*)]
    \item $\kappa(\fq)$ and $\kappa(\fq')$ both have more than 3 elements;
    \item The residue representations $\widetilde\rho_{\fq} \colon \Gamma_K \to {\rm GL}_{\kappa(\fq)}(\psi[\fq])$ and $\widetilde\rho'_{\fq'} \colon \Gamma_K \to {\rm GL}_{\kappa(\fq')}(\psi'[\fq'])$ are both surjective;
    \item $(\widetilde\rho_\fq,\widetilde\rho'_{\fq'})(\Gamma_K^{\rm der}) \neq {\rm SL}_{\kappa(\fq)}(\psi[\fq]) \times {\rm SL}_{\kappa(\fq')}(\psi'[\fq'])$.
\end{enumerate}
Then there exists an isomorphism of residue fields $\sigma \colon \kappa(\fq) \xrightarrow{\sim} \kappa(\fq')$ and a character $\chi \colon \Gamma_K \to \kappa(\fq')^\times$ such that one of the following two cases holds:

Case 1: There exists a $\sigma$-semilinear isomorphism $f \colon \psi[\fq] \xrightarrow{\sim} \psi'[\fq']$ satisfying
\begin{equation}\label{eqnwildetilde}
    \widetilde\rho'_{\fq'}(\gamma) = \chi(\gamma) \cdot f \circ \widetilde\rho_\fq(\gamma) \circ f^{-1} \qquad (\gamma \in \Gamma_{K}).
\end{equation}

Case 2: There exists a $\sigma$-semilinear isomorphism $f \colon \psi[\fq]^\vee \xrightarrow{\sim} \psi'[\fq']$ satisfying
\begin{equation}\label{eqnwildetilde1}
    \widetilde\rho'_{\fq'}(\gamma) = \chi(\gamma) \cdot f \circ \widetilde\rho_\fq(\gamma^{-1})^\vee \circ f^{-1} \qquad (\gamma \in \Gamma_{K}).
\end{equation}
\end{prop}

\begin{proof}
By assumption (b), $(\widetilde\rho_\fq,\widetilde\rho'_{\fq'})(\Gamma_K)$ maps onto each factor ${\rm GL}_{\kappa(\fq^\flat)}(\psi^\flat[\fq^\flat])$. Thus, by Goursat's lemma, there exist normal subgroups $N\triangleleft  {\rm GL}_{\kappa(\fq)}(\psi[\fq]) $, $N'\triangleleft {\rm GL}_{\kappa(\fq')}(\psi'[\fq'])$, and  
an isomorphism $\alpha:{\rm GL}_{\kappa(\fq)}(\psi[\fq])/N\xrightarrow{\sim} {\rm GL}_{\kappa(\fq')}(\psi'[\fq'])/N'$ such that
\begin{eqnarray}\label{eqnrhoqrhoqgamma}
(\widetilde\rho_\fq,\widetilde\rho'_{\fq'})(\Gamma_K)=\{(g,g')\in{\rm GL}_{\kappa(\fq)}(\psi[\fq])\times{\rm GL}_{\kappa(\fq')}(\psi'[\fq'])\mid \alpha(gN)=g'N'\}.
\end{eqnarray}

By \cite[Proposition~2.2]{PR}, $N^\flat$ either contains ${\rm SL}_{\kappa(\fq^\flat)}(\psi^\flat[\fq^\flat])$ or is contained in the center of ${\rm GL}_{\kappa(\fq^\flat)}(\psi^\flat[\fq^\flat])$, depends on whether ${\rm GL}_{\kappa(\fq^\flat)}(\psi^\flat[\fq^\flat])/N^\flat $ is commutative or not.

If the former case holds for $N$, then so does $N'$. We have
$(\widetilde\rho_\fq,\widetilde\rho'_{\fq'})(\Gamma_K)$ contains ${\rm SL}_{\kappa(\fq)}(\psi[\fq])\times{\rm SL}_{\kappa(\fq')}(\psi'[\fq'])$.
By \cite[Corollary 4.3]{BMS}, ${\rm SL}_{\kappa(\fq)}(\psi[\fq])\times{\rm SL}_{\kappa(\fq')}(\psi'[\fq'])$ is perfect. Then
$(\widetilde\rho_\fq,\widetilde\rho'_{\fq'})(\Gamma_K^{\rm der})$ contains ${\rm SL}_{\kappa(\fq)}(\psi[\fq])\times{\rm SL}_{\kappa(\fq')}(\psi'[\fq'])$, which contradicts to condition (c).

Thus $N^\flat$ is contained in the center of ${\rm GL}_{\kappa(\fq^\flat)}(\psi^\flat[\fq^\flat])$ for each $\flat$. Modulo the respective centers, $\alpha$ induces an isomorphism $\bar\alpha:{\rm PGL}_{\kappa(\fq)}(\psi[\fq])\xrightarrow{\sim} {\rm PGL}_{\kappa(\fq')}(\psi'[\fq']).$  It then follows from \cite{Die, Ste} that there exists a field isomorphism $\sigma:\kappa(\fq)\xrightarrow{\sim}\kappa(\fq')$, along with either a $\sigma$-semilinear isomorphism $f:\psi[\fq]\xrightarrow{\sim}\psi'[\fq']$ or a $\sigma$-semilinear isomorphism $f:\psi[\fq]^\vee\xrightarrow{\sim}\psi'[\fq']$, such that $\bar\alpha$ is induced by the map $g\mapsto f\circ g\circ f^{-1}$ or $g\mapsto f\circ(g^\vee)^{-1}\circ f^{-1}$, respectively. For any $\gamma\in\Gamma_K$, equation \eqref{eqnrhoqrhoqgamma} implies that there exists a scalar $\chi(\gamma)\in\kappa(\fq')^\times$ such that either $\widetilde\rho'_{\fq'}(\gamma)=\chi(\gamma)f\circ\widetilde\rho_{\fq}(\gamma)\circ f^{-1}$ or $\widetilde\rho'_{\fq'}(\gamma)=\chi(\gamma)f\circ\widetilde\rho_{\fq}(\gamma^{-1})^\vee\circ f^{-1}$. This completes the proof.
\end{proof}

\begin{lem}\label{lem:lift}
Let $L$ and $L'$ be global function fields of characteristic $p$, and let $(x_i)_{i\in I}$ and $(x_i')_{i\in I}$ be families of elements such that
\[
L=\F_p(x_i:i\in I),\qquad L'=\F_p(x_i':i\in I).
\]
Let $\Lambda$ be an infinite set, and let $(v_\lambda)_{\lambda\in\Lambda}$ and $(v_\lambda')_{\lambda\in\Lambda}$ be families of pairwise distinct places of $L$ and $L'$, respectively. Suppose that for each $\lambda\in\Lambda$, there is an isomorphism of residue fields
\[
s_\lambda:\kappa(v_\lambda)\xrightarrow{\sim}\kappa(v_\lambda').
\]
Assume that for every $i\in I$, the elements $x_i$ and $x_i'$ are integral at $v_\lambda$ and $v_\lambda'$, respectively, and satisfy
\begin{equation}\label{eqn:slambda}
    s_\lambda(x_i\bmod v_\lambda)=x_i'\bmod v_\lambda'
\end{equation}
for all but finitely many $\lambda\in\Lambda$.

Then there exists a unique field isomorphism $s:L\xrightarrow{\sim}L'$ such that $s(x_i)=x_i'$ for every $i\in I$. Moreover, for all but finitely many $\lambda\in\Lambda$, we have 
$
s(\mathcal O_{v_\lambda})=\mathcal O_{v_\lambda'}, 
$
where $\sO_{v_\lambda}$ and $\sO_{v_\lambda'}$ denotes the corresponding valuation ring. In particular, $s$ induces an isomorphism of complete valued fields $L_{v_\lambda}\xrightarrow{\sim}L'_{v_\lambda'}$ for all but finitely many $\lambda\in\Lambda$.
\end{lem}

\begin{proof}
We first construct the field isomorphism. Suppose that
\[
R(x_{i_1},\ldots,x_{i_r})=0
\]
for some $R\in\F_p[T_1,\ldots,T_r]$. By \eqref{eqn:slambda}, we have
\[
R(x_{i_1}',\ldots,x_{i_r}')\equiv0\pmod{v_\lambda'}
\]
for all but finitely many $\lambda\in\Lambda$. Since the places $v_\lambda'$ are pairwise distinct, a nonzero element of $L'$ cannot vanish at infinitely many of them. Thus $R(x_{i_1}',\ldots,x_{i_r}')=0$. Applying the same argument to $s_\lambda^{-1}$ gives the converse. Consequently, the assignment $x_i\mapsto x_i'$ induces an isomorphism
\[
\F_p[x_i:i\in I]\xrightarrow{\sim}\F_p[x_i':i\in I],
\]
which extends uniquely to a field isomorphism $s:L\xrightarrow{\sim}L'$.

It remains to prove the compatibility with the places. Since $L$ is finitely generated over $\F_p$, we may choose finitely many indices $i_1,\ldots,i_r$ such that $L=\F_p(x_{i_1},\ldots,x_{i_r})$. Set
\[
D=\F_p[x_{i_1},\ldots,x_{i_r}]\quad \text{ and }\quad D'=s(D).
\]
Let $\widetilde D$ be the integral closure of $D$ in $L$. Since $D$ is a one-dimensional finitely generated $\F_p$-domain, $\widetilde D$ is finite over $D$. Moreover, $D$ and $\widetilde D$ have the same fraction field, so there exists a nonzero element $h\in D$ such that
\[
D[1/h]=\widetilde D[1/h].
\]
Thus $D[1/h]$ is a normal one-dimensional domain, and so is its image $D'[1/s(h)]$ under $s$.

After discarding finitely many $\lambda\in\Lambda$, we may assume that all the chosen generators are integral at the corresponding places, that $h$ and $s(h)$ are units in $\mathcal O_{v_\lambda}$ and $\mathcal O_{v_\lambda'}$, respectively, and that the congruences in \eqref{eqn:slambda} hold for all the chosen generators. Reduction then induces homomorphisms
\[
D[1/h]\longrightarrow\kappa(v_\lambda),\qquad D'[1/s(h)]\longrightarrow\kappa(v_\lambda').
\]
Let $\mathfrak m_\lambda$ and $\mathfrak m_\lambda'$ denote their respective kernels. These are maximal ideals, and the compatibility in \eqref{eqn:slambda} implies that $s(\mathfrak m_\lambda)=\mathfrak m_\lambda'$. Since $D[1/h]$ and $D'[1/s(h)]$ are normal one-dimensional Noetherian domains, their localizations at these maximal ideals are discrete valuation rings. Moreover, these localizations are dominated by $\mathcal O_{v_\lambda}$ and $\mathcal O_{v_\lambda'}$, respectively. As a discrete valuation ring admits no proper valuation-ring extension with the same fraction field that dominates it, we obtain
\[
D[1/h]_{\mathfrak m_\lambda}=\mathcal O_{v_\lambda},\qquad D'[1/s(h)]_{\mathfrak m_\lambda'}=\mathcal O_{v_\lambda'}.
\]
Since $s(\mathfrak m_\lambda)=\mathfrak m_\lambda'$, it follows that
\[
s(\mathcal O_{v_\lambda})=\mathcal O_{v_\lambda'}
\]
for all but finitely many $\lambda\in\Lambda$. Consequently, $s$ extends uniquely to an isomorphism of complete valued fields $L_{v_\lambda}\xrightarrow{\sim}L'_{v_\lambda'}$ for all such $\lambda$.
\end{proof}

\begin{prop}\label{prop:residuerepren}
Assume that $\psi$ and $\psi'$ both have rank at least two, and that $\phi$ and $\phi'$ are not geometrically isogenous up to any Frobenius twist. Then there exist finite subsets $I\subset{\rm Max}\,B$ and $I'\subset{\rm Max}\,B'$ such that
\[
(\widetilde\rho_\fq,\widetilde\rho'_{\fq'})(\Gamma_K^{\rm der})={\rm SL}_{\kappa(\fq)}(\psi[\fq])\times{\rm SL}_{\kappa(\fq')}(\psi'[\fq'])
\]
for every pair $(\fq,\fq')\in({\rm Max}\,B\setminus I)\times({\rm Max}\,B'\setminus I')$. Equivalently, condition (c) of Proposition~\ref{prop:residual} fails for every such pair.
\end{prop}
\begin{proof}
We prove it by contradiction. Suppose the assertion fails. By Theorem~\ref{Thm: Pink_open_image}, conditions (a) and (b) of Proposition~\ref{prop:residual} fail for only finitely many $\fq\in{\rm Max}\,B$ and $\fq'\in{\rm Max}\,B'$. Hence inductively, we can find an infinite sequence $(\fq_i,\fq_i')_{i\in\N}$ with $\fq_i\in{\rm Max}\backslash \{\fq_0, \dots, \fq_{i-1}\}$, and $\fq_i'\in{\rm Max}\backslash \{\fq_0', \dots, \fq_{i-1}'\}$, each satisfying all conditions of Proposition~\ref{prop:residual}. Hence there are field isomorphisms $\sigma_i:\kappa(\fq_i)\xrightarrow{\sim}\kappa(\fq_i')$ with $a(\widetilde\rho'_{\fq_i'}(\gamma))=\sigma_i\bigl(a(\widetilde\rho_{\fq_i}(\gamma))\bigr)$ for all $\gamma\in\Gamma_K$, since $a$ is invariant under scalars, conjugation and inverse transpose.

Let $X={\rm Spec}\,R$ be as in the proof of Theorem~\ref{11}. For $x\in|X|$, put $a_x=a(\rho_\fq({\rm Frob}_x))$ and $a'_x=a(\rho'_{\fq'}({\rm Frob}_x))$ for any $\fq\neq\fq_x$, $\fq'\neq\fq_x'$. By Theorem~\ref{2}\,(1),(2) we have $a_x\in E$ and $a'_x\in E'$, both independent of the choice of $\fq$ and $\fq'$. Moreover $a_x$ is integral at every $\fq\neq\fq_x$ with $a(\widetilde\rho_\fq({\rm Frob}_x))=a_x\bmod\fq$, and similarly for $a'_x$. Hence for each $x$,
\[
\sigma_i(a_x\bmod\fq_i)=a'_x\bmod\fq_i'\qquad\text{for all but finitely many }i.
\]
By Theorem~\ref{2}\,(3) we have 
$$E=\F_p(a_x:x\in|X|) \qquad E'=\F_p(a'_x:x\in|X|).$$
Lemma~\ref{lem:lift} then gives a field isomorphism $\sigma:E\xrightarrow{\sim}E'$ with $\sigma(a_x)=a'_x$ for all $x$, and $\sigma(\sO_{\fq_i})=\sO_{\fq_i'}$ for all but finitely many $i$.

Fix such an $i$ and write $\fq=\fq_i$, $\fq'=\fq_i'$. Then $\sigma$ extends to an isomorphism $\widehat\sigma:E_\fq\xrightarrow{\sim}E'_{\fq'}$, and $\widehat\sigma(a_x)=a'_x$ for all $x\in|X|$ away from $\fq,\fq'$. By the Chebotarev density theorem and continuity,
\[
\widehat\sigma\bigl(a(\rho_\fq(\gamma))\bigr)=a(\rho'_{\fq'}(\gamma))\qquad(\gamma\in\Gamma_K).
\]
This algebraic relation implies that $(\rho_{\mathfrak q},\rho'_{\mathfrak q'})(\Gamma_K^{\rm der})$ is not open in $\SL_{B_\fq}(T_\fq\psi)\times\SL_{B'_{\fq'}}(T_{\fq'}\psi')$. 
Applying Theorem \ref{thmmainsl} to the given pair $(\fq,\fq')$, then $\phi$ and $\phi'$ are geometrically isogenous up to some Frobenius twist.
It contradicts our assumptions, which yields the proposition.
\end{proof}

\section{Proof of the main theorem and torsion finiteness}\label{Sect: final_pf}
We are now in a position to prove Theorem~\ref{thmmain} by combining the results of the preceding three sections. Theorem~\ref{thmmainsl} establishes pairwise local openness for the derived Galois image, while Proposition~\ref{prop:residuerepren} gives the full product of the residual special linear groups outside finitely many primes. Together with the determinant openness established in Theorem~\ref{Thm: det_open}, Lemma~\ref{lem:det-comm} reduces the adelic commensurability assertion to proving that $(\rho,\rho')(\Gamma_K^{\rm der})$ is open in $U^{\rm der}$. This openness will follow from the profinite-group criterion in Theorem~\ref{thm:mainkey}, which we establish below. Finally, we deduce Theorem~\ref{Cor: torsion_finite} as a consequence.

\subsection{Openness of subgroups in infinite products of profinite groups}
Throughout this subsection, we use the following two properties of a profinite group $G$.

\begin{itemize}
    \item[(SC)] (\emph{strong completeness}) Every abstract subgroup of finite index in $G$ is open.
    \item[(FAb)] (\emph{finite abelianization}) For every open subgroup $U$ of $G$, the abstract abelianization $U/[U,U]$ is finite.
\end{itemize}
Note that if $G$ satisfies (SC), then so does every subgroup of finite index in $G$. We begin with two elementary lemmas.

\begin{lem}
\label{lem:sc-product}
If $G_1,\ldots,G_n$ satisfy $({\rm SC})$, then
$G_1\times\cdots\times G_n$ satisfies $({\rm SC})$.
\end{lem}

\begin{proof}
Let $H$ be a subgroup of $G_1\times\cdots\times G_n$ of finite index.  Regard $G_i$ as the $i$-th coordinate subgroup of the product.  Then
$[G_i:H\cap G_i]\leq [G_1\times\cdots\times G_n:H]<\infty.$
Property $({\rm SC})$ implies that $H\cap G_i$ is open in $G_i$.  Consequently, $H\supset (H\cap G_1)\times\cdots\times(H\cap G_n)$ and hence is open. 
\end{proof}

\begin{lem}
\label{lem:finite-pairwise}
Suppose $n\geq 2$ and let $G=\prod\limits_{i=1}^nG_i$ be a finite product of profinite groups $G_i$ satisfying $({\rm SC})$ and $({\rm FAb})$,
and let $H$ be an arbitrary (not necessarily closed) subgroup of $G$. If the projection image of $H$ in $G_i\times G_j$ is open for all $i\neq j$, then $H$ is open in $G$.
\end{lem}

\begin{proof}
We argue by induction on $n$. The case $n=2$ is immediate. Suppose $n\geq3$, and write $G=G'\times G_n$, where $G'=\prod\limits_{i=1}^{n-1}G_i$. Let $H'$ be the image of $H$ in $G'$. By the induction hypothesis, $H'$ is open in $G'$. Since the projection of $H$ to $G_n$ is open, and open subgroups inherit properties (SC) and (FAb), we may replace $G_n$ by this image and assume that $H$ maps onto $G_n$. By Goursat's lemma, there exist a group $Q$ and surjective homomorphisms $\alpha:H'\to Q$ and $\beta:G_n\to Q$ such that $H=H'\times_Q G_n$. Since $H'$ has finite index in $G'$, it suffices to prove that $Q$ is finite. Indeed, this would imply that $H$ has finite index in $G$, and hence is open by Lemma~\ref{lem:sc-product} and property (SC).

Since $H'$ is open in $G'$, there exist open subgroups $V_i\leq G_i$ such that $V:=\prod\limits_{i=1}^{n-1}V_i\subseteq H'$. For each $i<n$, regard $V_i$ as a subgroup of $V$ and put $Q_i:=\alpha(V_i)$. These subgroups commute pairwise, and their product $Q':=\prod\limits_{i=1}^{n-1}Q_i=\alpha(V)$ has finite index in $Q$.

For each $i<n$, put $Q_{\widehat i}:=\prod\limits_{j<n,\;j\neq i}Q_j$ and $K_i:=\ker(H'\to G_i)$. Since the image of $H$ in $G_i\times G_n$ is open, $\alpha(K_i)$ has finite index in $Q$. Moreover, since $V$ has finite index in $H'$, the subgroup $K_i\cap V$ has finite index in $K_i$. Thus $Q_{\widehat i}=\alpha(K_i\cap V)$ has finite index in $Q$. Consequently,
$
C:=\bigcap\limits_{i=1}^{n-1}Q_{\widehat i}
$
has finite index in $Q$. Since $C$ centralizes each $Q_i$ and is contained in $Q'=\prod\limits_{i=1}^{n-1}Q_i$, it is abelian.

Now $\beta^{-1}(C)$ has finite index in $G_n$ and is therefore open by (SC). Since $G_n$ satisfies (FAb), the abelianization $\beta^{-1}(C)/[\beta^{-1}(C),\beta^{-1}(C)]$ is finite. As this abelianization surjects onto $C$, the group $C$ is finite. Since $C$ has finite index in $Q$, we conclude that $Q$ is finite, as desired.
\end{proof}

A finite group $T$ is called a \emph{continuous subquotient} of a profinite group $G$ if there are closed subgroups $N\triangleleft H\leq G$ such that $H/N\cong T$. We generalize Lemma \ref{lem:finite-pairwise} as follows.

\begin{thm}\label{thm:mainkey}
Let $G=\prod\limits_{i\in I}G_i$ be a product of profinite groups, each satisfying $({\rm SC})$ and $({\rm FAb})$. Assume $|I|\geq 2$. Let $H\leq G$ be a closed subgroup. Suppose there is a finite subset $J\subset I$ such that:
\begin{enumerate}[label=(\arabic*)]
  \item for all distinct $i,j\in I$, the image of $H$ in $G_i\times G_j$ is open;
  \item for every $i\notin J$, the projection $H\to G_i$ is surjective, and $G_i$ has exactly one maximal proper closed normal subgroup $M_i$. Moreover, $T_i:=G_i/M_i$ is a finite nonabelian group.
  \item for distinct $i,j\notin J$,$H$ maps onto  $T_i\times T_j$;
  \item for all but finitely many $i\notin J$, the group $T_i$ is not a continuous subquotient of $\prod\limits_{j\in J}G_j$.
\end{enumerate}
Then $H$ is open in $G$.
\end{thm}

\begin{proof}
For a subset $P\subseteq I$, write $G_P=\prod\limits_{i\in P}G_i$ and let $H_P$ be the image of $H$ in $G_P$. If $P=\{i,j\}$ we simply write $G_{ij}$ and $H_{ij}$. Since $H$ is closed in the compact group $G$, each $H_P$ is closed in $G_P$. We prove the theorem in four steps.

\underline{Step (i)}. \emph{Prove that $H_{ij}=G_{ij}$ for distinct $i,j\notin J$.}

Suppose $H_{ij}\neq G_{ij}$. By (2), $H_{ij}$ maps onto both $G_i$ and $G_j$, so the topological Goursat lemma gives proper closed normal subgroups $N_i\triangleleft G_i$, $N_j\triangleleft G_j$ and a topological isomorphism $\sigma:G_i/N_i\xrightarrow{\sim}G_j/N_j$ such that
\[
  H_{ij}=\{(g_i,g_j)\in G_i\times G_j\mid \sigma(g_iN_i)=g_jN_j\}.
\]
By (2), $N_i\leq M_i$ and $N_j\leq M_j$, and $M_i/N_i$ (resp.\ $M_j/N_j$) is the unique maximal proper closed normal subgroup of $G_i/N_i$ (resp.\ $G_j/N_j$). Hence $\sigma$ carries $M_i/N_i$ onto $M_j/N_j$ and induces an isomorphism $\overline{\sigma}:T_i\xrightarrow{\sim}T_j$. Consequently the image of $H$ in $T_i\times T_j$ is contained in the graph of $\overline{\sigma}$, a proper subgroup, which contradicts (3). This proves Step (i).

\underline{Step (ii)}. \emph{Prove $H_P=G_P$ for any finite subset $P$ of $I\backslash J$.}

We argue by induction on $|P|$. The case $|P|= 2$ follows from Step (i). Suppose $|P|\geq 3$, fix $i_0\in P$, and put $P'=P\setminus\{i_0\}$. By induction, $H_P$ maps onto $G_{P'}$, and by (2) it maps onto $G_{i_0}$. Since $H_P$ is closed in $G_P$, the topological Goursat lemma gives a profinite group $Q$ and continuous surjections $\alpha:G_{P'}\twoheadrightarrow Q$ and $\beta:G_{i_0}\twoheadrightarrow Q$ such that
\[
  H_P=\{(a,b)\in G_{P'}\times G_{i_0}\mid \alpha(a)=\beta(b)\}.
\]
It suffices to show that $Q=1$.

For $i\in P'$, write $Q_i=\alpha(G_i)$. These are closed subgroups of $Q$ that commute pairwise and generate $Q$. Fix $i_1\in P'$. By Step (i), $H_{i_1i_0}=G_{i_1}\times G_{i_0}$, so for every $b\in G_{i_0}$ there is $a\in G_{P'}$ with trivial $i_1$-th coordinate and $\alpha(a)=\beta(b)$. Hence 
\[
Q=\beta(G_{i_0})=\prod\limits_{i\in P'\setminus\{i_1\}}Q_i,
\] 
In particular $Q_{i_1}$ centralizes $Q$. Let $i_1$ run over $P'$, and recall that $Q$ is generated by the $Q_i$. We have that $Q$ is abelian.

On the other hand, by assumption (2), $G_{i_0}$ has a unique nontrivial minimal quotient, which is nonabelian. Since $Q$ is an abelian quotient of $G_{i_0}$, we must have $Q=1$. This proves Step (ii).

\underline{Step (iii)}. \emph{Prove $H_{I\backslash J}=G_{I\backslash J}$.}

In Step (ii), we have proved $H_P=G_P$ for any finite subset $P\subset I\backslash J$. It follows that $H_{I\backslash J}$ is dense in $G_{I\backslash J}$. On the other hand, $H_{I\backslash J}$ is compact and hence closed in $G_{I\backslash J}$, we have  $H_{I\backslash J}=G_{I\backslash J}$.

\underline{Step (iv)}. \emph{Prove $H$ is open in $G$.} 

By (1) and Lemma~\ref{lem:finite-pairwise}, $H_J$ is open in $G_J$. By Step~(iii), $H_{I\setminus J}=G_{I\setminus J}$. Since $H$ is closed and maps onto both $H_J$ and $G_{I\setminus J}$, the topological Goursat lemma gives a profinite group $C$ and continuous surjections
\[
  \gamma:H_J\twoheadrightarrow C,\qquad \delta:G_{I\setminus J}\twoheadrightarrow C
\]
such that $H=H_J\times_C G_{I\setminus J}$. For $i\notin J$, put $C_i=\delta(G_i)$.

Suppose $C_i\neq 1$. Then $\ker(\delta|_{G_i})$ is a proper closed normal subgroup of $G_i$, so $\ker(\delta|_{G_i})\leq M_i$ by (2), and hence $T_i=G_i/M_i$ is a quotient of $C_i$. Now $\gamma^{-1}(C_i)$ is a closed subgroup of $H_J$, and the composite $\gamma^{-1}(C_i)\twoheadrightarrow C_i\twoheadrightarrow T_i$ shows that $T_i$ is a continuous subquotient of $G_J$. By (4), the set
\[
  J'=\{i\in I\setminus J\mid C_i\neq 1\}
\]
is finite.

Since $\delta$ is trivial on each $G_i$ with $i\notin J\cup J'$, by continuity $\delta$ is trivial on $G_{I\setminus(J\cup J')}$. Hence $\delta$ factors through a surjection $G_{J'}\twoheadrightarrow C$, and
\[
  H=H_J\times_C G_{I\setminus J}=(H_J\times_C G_{J'})\times G_{I\setminus(J\cup J')}=H_{J\cup J'}\times G_{I\setminus(J\cup J')}.
\]
By (1) and Lemma~\ref{lem:finite-pairwise}, $H_{J\cup J'}$ is open in $G_{J\cup J'}$. Therefore $H$ is open in $G$.
\end{proof}

\begin{lem}\label{lem:SC}
Let $\mathcal O$ be the valuation ring of a nonarchimedean local field, with maximal ideal $\mathfrak m$ and residue field $k$, and let $r\geq 2$ be an integer. 
\begin{enumerate}
  \item[(1)] ${\rm SL}_r(\mathcal O)$ satisfies $({\rm SC})$ and $({\rm FAb})$.
\end{enumerate}
Assume moreover that $(r,|k|)\neq(2,2),(2,3)$. Then:
\begin{enumerate}
  \item[(2)] ${\rm SL}_r(\mathcal O)$ is perfect.
  \item[(3)] Let $M$ be the kernel of the natural surjection ${\rm SL}_r(\mathcal O)\twoheadrightarrow{\rm PSL}_r(k)$. Then ${\rm PSL}_r(k)$ is a finite nonabelian simple group, and $M$ is the unique maximal proper closed normal subgroup of ${\rm SL}_r(\mathcal O)$.
\end{enumerate}
\end{lem}

\begin{proof}
Let $v$ be the normalized discrete valuation of $\mathcal O$. For each $t\geq1$, let
\[
K_t:=\ker\bigl({\rm SL}_r(\mathcal O)\to
{\rm SL}_r(\mathcal O/\mathfrak m^t)\bigr).
\]
Fix $t\geq1$ and choose $u\in1+\mathfrak m^t$ such that $u^2\neq1$. Write $s=v(u^2-1)$. For $i\neq j$, let $d_{ij}$ be the diagonal matrix with entries $u$ and $u^{-1}$ in positions $i$ and $j$, respectively, and $1$ elsewhere. Then
\[
[d_{ij},1+xE_{ij}]
=1+(u^2-1)xE_{ij}
\qquad(x\in\mathfrak m^t).
\]
Thus, for every $i\neq j$ and $x\in\mathfrak m^{t+s}$, we have
$
I_r+xE_{ij}\in[K_t,K_t]
$. 

Fix a uniformizer $\pi$ of $\mathcal O$. For any $h\in1+\mathfrak m^{2(t+s)}$, choose $x=\pi^{t+s}$ and $y=(h-1)/\pi^{t+s}$. Then $x,y\in\mathfrak m^{t+s}$ and $1+xy=h$. The identity
\[
(I_r+xE_{ij})(I_r+yE_{ji})
\left(I_r-\frac{x}{1+xy}E_{ij}\right)
\left(I_r-y(1+xy)E_{ji}\right)=I_r+(h-1)E_{ii}+(h^{-1}-1)E_{jj}
\]
shows that the diagonal matrix with entries $h$ and $h^{-1}$ in any two distinct positions, and $1$ elsewhere, belongs to $[K_t,K_t]$.

By Gaussian elimination, every element of $K_{2(t+s)}$ is a product of elementary matrices with entries in $\mathfrak m^{2(t+s)}$ and diagonal matrices of the form considered above. So we obtain
\[
K_{2(t+s)}\subseteq[K_t,K_t].
\]
Consequently, the commutator subgroup of every open subgroup of ${\rm SL}_r(\mathcal O)$ is open, proving (FAb). In particular, $K_1$ is a pro-$p$ group with finite abelianization. By the Burnside basis theorem for pro-$p$ groups \cite[Propositions 1.13 and 1.14]{DS}, $K_1$ is topologically finitely generated. Since $K_1$ has finite index in ${\rm SL}_r(\mathcal O)$, the latter is also topologically finitely generated. Property (SC) now follows from \cite[Theorem 1.1]{NS}. This proves (1). 

Part~(2) follows from \cite[Lemma 2.2.3]{Ea}. To prove (3), let $N$ be any maximal proper closed normal subgroup of $G:={\rm SL}_r(\mathcal O)$. Since $G$ is profinite, $N$ is open and $G/N$ is finite simple. By part~(2), $G/N$ is perfect and hence nonabelian. The image of $K_1$ in $G/N$ is a normal $p$-subgroup, so it is trivial. Thus the quotient map factors through ${\rm SL}_r(k)$. Since the center of ${\rm SL}_r(k)$ maps trivially into $G/N$, it further factors through ${\rm PSL}_r(k)$. By \cite[Theorem 2.10]{Ca}, ${\rm PSL}_r(k)$ is simple, so the induced surjection is an isomorphism. Therefore $N=M$, proving (3).
\end{proof}

\subsection{Proof of the main theorems}
We retain the notations and assumptions introduced in \S\,\ref{Sect: main_thm} and \S\,\ref{Sect: simplifications}. 
As noted in Remark~\ref{Ufp}, one can express $U_\fp$ as a subgroup of $\prod\limits_{\fq\mid\fp}{\rm GL}_{B_\fq}(T_\fq\psi)\times\prod\limits_{\fq'\mid\fp}{\rm GL}_{B'_{\fq'}}(T_{\fq'}\psi')$. Specifically, it consists of the elements $(g_\fq, g'_{\fq'})$ such that for every maximal ideal $\fp'$ of $A'$ lying above $\fp$, the following local component equality holds:
\[
  \prod_{\fq\mid\fp'}{\rm N}_{E_\fq/F'_{\fp'}}(\det(g_\fq)) = \prod_{\fq'\mid\fp'}{\rm N}_{E'_{\fq'}/F'_{\fp'}}(\det(g'_{\fq'})).
\]

\begin{lem}\label{lem:det-comm}
Let $G$ be a group, let $N\triangleleft G$, and let $H,U\leq G$ with $N\subseteq U$. If $H\cap N$ has finite index in $N$, and the images of $H$ and $U$ in $G/N$ are commensurable, then $H$ and $U$ are commensurable.
\end{lem}

\begin{proof}
Let $\pi:G\to G/N$ be the quotient map. Since $\pi(H)$ and $\pi(U)$ are commensurable, so are their preimages $HN$ and $U$. Moreover,
\[
[HN:H]=[N:H\cap N]<\infty.
\]
Thus $H$ and $U$ are commensurable.
\end{proof}

\begin{rem}\label{Rmk: part_3}
First, if we apply Lemma~\ref{lem:det-comm} by taking
\[
\begin{aligned}
G&:=\prod_{\fq\in{\rm Max}\,B}{\rm GL}_{B_\fq}(T_\fq\psi)
\times\prod_{\fq'\in{\rm Max}\,B'}{\rm GL}_{B'_{\fq'}}(T_{\fq'}\psi'), & U&:=\prod_{\fp\in{\rm Max}\,A}U_\fp,\\
N&:=\prod_{\fq\in{\rm Max}\,B}{\rm SL}_{B_\fq}(T_\fq\psi)
\times\prod_{\fq'\in{\rm Max}\,B'}{\rm SL}_{B'_{\fq'}}(T_{\fq'}\psi'), & H&:=(\rho,\rho')(\Gamma_K).\\
\end{aligned}
\]
Here $N=\ker(\det:G\to\det(G))$ is a normal subgroup of $G$ contained in $U$.
By Theorem~\ref{Thm: det_open}, the images of $H$ and $U$ in $G/N$ are commensurable. Moreover, notice that $
(\rho,\rho')(\Gamma_K^{\rm der})\subseteq H\cap N$. So in order to prove that $H$ and $U$ are commensurable, i.e., part (3) of Theorem~\ref{thmmain}, it suffices to prove that $(\rho,\rho')(\Gamma_K^{\rm der})$ is open in $N$.
\end{rem}

\begin{proof}[Proof of Theorem~\ref{thmmain}]
Suppose that one of $\phi$ and $\phi'$ has potential complex multiplication, say $\phi'$. Then $\psi'$ has rank one, so $\rho'(\Gamma_K^{\rm der})$ is trivial and
\begin{equation}\label{Eqn: U_der_CM}
    N=\prod_{\fq\in{\rm Max}\,B}{\rm SL}_{B_\fq}(T_\fq\psi).
\end{equation}
By Theorem~\ref{Thm: Pink_open_image}, the image
\[
(\rho,\rho')(\Gamma_K^{\rm der})
=(\rho_\fq)_{\fq\in{\rm Max}\,B}(\Gamma_K^{\rm der})
\]
is open in $N$. By Remark~\ref{Rmk: part_3}, $(\rho,\rho')(\Gamma_K)$ is commensurable with $U$. It follows immediately that the three statements of Theorem~\ref{thmmain} all hold  whenever one of $\phi$ and $\phi'$ has potential complex multiplication. In the rest of the proof, we assume that neither $\phi$ nor $\phi'$ has potential complex multiplication. 

The implication (3) $\Rightarrow$ (2) is obvious. Now suppose (2) holds. Then $(\rho_\fp,\rho'_\fp)(\Gamma_K)$ contains some open subgroup $V$ of $\prod\limits_{\fq\mid\fp}{\rm SL}_{B_\fq}(T_\fq\psi)\times\prod\limits_{\fq'\mid\fp}{\rm SL}_{B'_{\fq'}}(T_{\fq'}\psi')$. Thus $(\rho_\fp,\rho'_\fp)(\Gamma_K^{\rm der})$ contains $[V,V]$, which is open in $\prod\limits_{\fq\mid\fp}{\rm SL}_{B_\fq}(T_\fq\psi)\times\prod\limits_{\fq'\mid\fp}{\rm SL}_{B'_{\fq'}}(T_{\fq'}\psi')$ by Lemma \ref{lem:SC}(1) and the property (FAb). In particular, $(\rho_\fq,\rho'_{\fq'})(\Gamma_K^{\rm der})$ has open image in ${\rm SL}_{B_\fq}(T_\fq\psi)\times\prod{\rm SL}_{B'_{\fq'}}(T_{\fq'}\psi')$
for every $\fq$ and $\fq'$ lying above the same $\fp\in {\rm Max}\,A$. Then (1) holds by Theorem~\ref{thmmainsl}. This proves (2)$\Rightarrow$ (1). It remains to show (1)$\Rightarrow$(3).

Suppose $\phi$ and $\phi'$ are not geometrically isogenous up to any Frobenius twist. By definition,
\begin{equation}\label{Eqn: U_der}
    N = \prod\limits_{\fq\in{\rm Max}\,B} {\rm SL}_{B_\fq}(T_\fq\psi) \times \prod\limits_{\fq'\in{\rm Max}\,B'} {\rm SL}_{B'_{\fq'}}(T_{\fq'}\psi').
\end{equation}
By Remark~\ref{Rmk: part_3}, it suffices to show that 
\begin{equation}\label{Eqn: main_goal}
    (\rho,\rho')(\Gamma_K^{\rm der}) \text{ is open in }N. 
\end{equation}
For this purpose, we consider $I\coloneqq {\rm Max}\,B\sqcup{\rm Max}\,B'$, so that $U^{\rm der}=\prod\limits_{\fq^\flat\in I} {\rm SL}_{B^\flat_{\fq^\flat}}(T_{\fq^\flat}\psi^\flat)$. It then reduces to verifying that all the conditions of Theorem~\ref{thm:mainkey} hold for $G=U^{\rm der}$ and $H=(\rho,\rho')(\Gamma_K^{\rm der})$.

First, each ${\rm SL}_{B^\flat_{\fq^\flat}}(T_{\fq^\flat}\psi^\flat)$ satisfies properties (SC) and (FAb) by Lemma~\ref{lem:SC}(1).

Theorem~\ref{Thm: Pink_open_image} implies that $(\rho^\flat_{\fq^\flat},\rho^\flat_{\fq_1^\flat})(\Gamma_K^{\rm der})$ is open in ${\rm SL}_{B^\flat_{\fq^\flat}}(T_{\fq^\flat}\psi^\flat)\times {\rm SL}_{B^\flat_{\fq_1^\flat}}(T_{\fq_1^\flat}\psi^\flat)$ for any distinct $\fq^\flat,\fq_1^\flat\in{\rm Max}\,B^\flat$, and that $(\widetilde\rho^\flat_{\fq^\flat},\widetilde\rho^\flat_{\fq_1^\flat})(\Gamma_K^{\rm der})={\rm SL}_{\kappa(\fq^\flat)}(\psi^\flat[\fq^\flat])\times {\rm SL}_{\kappa(\fq_1^\flat)}(\psi^\flat[\fq_1^\flat])$ for any distinct $\fq^\flat,\fq_1^\flat$ outside some finite subset of ${\rm Max}\,B^\flat$. On the other hand, by Theorem~\ref{thmmainsl}, $(\rho_\fq,\rho'_{\fq'})(\Gamma_K^{\rm der})$ is open in ${\rm SL}_{B_{\fq}}(T_{\fq}\psi)\times {\rm SL}_{B'_{\fq'}}(T_{\fq'}\psi')$ for any $(\fq,\fq')\in{\rm Max}\,B\times{\rm Max}\,B'$. By Proposition~\ref{prop:residuerepren}, $(\widetilde\rho_\fq,\widetilde\rho'_{\fq'})(\Gamma_K^{\rm der})={\rm SL}_{\kappa(\fq)}(\psi[\fq])\times {\rm SL}_{\kappa(\fq')}(\psi'[\fq'])$ for any pair $(\fq,\fq')$ with $\fq$ outside some finite subset of ${\rm Max}\,B$ and $\fq'$ outside some finite subset of ${\rm Max}\,B'$. This proves that $(G,H)$ satisfies conditions (1) and (3) in Theorem~\ref{thm:mainkey}.

Theorem~\ref{Thm: Pink_open_image} also implies that $\rho^\flat_{\fq^\flat}(\Gamma_K^{\rm der})$ maps onto ${\rm SL}_{B^\flat_{\fq^\flat}}(T_{\fq^\flat}\psi^\flat)$ for all but finitely many $\fq^\flat\in{\rm Max}\,B^\flat$. By Lemma~\ref{lem:SC}, for all but finitely many $\fq^\flat$, the group ${\rm SL}_{B^\flat_{\fq^\flat}}(T_{\fq^\flat}\psi^\flat)$ has a unique maximal proper closed normal subgroup whose quotient is a finite nonabelian simple group. This proves that $(G,H)$ satisfies condition (2) in Theorem~\ref{thm:mainkey}.

In other words, there exists a finite subset $J$ of $I={\rm Max}\,B\sqcup{\rm Max}\,B'$ such that $(G,H)$ satisfies conditions (2) and (3) in Theorem~\ref{thm:mainkey}. It remains to verify condition (4).

Let $N\leq G_J$ be closed, and let $u:N\twoheadrightarrow {\rm PSL}_{\kappa(\fq_0^\flat)}(\psi^\flat[\fq_0^\flat])$ be a continuous epimorphism for some $\fq_0^\flat\notin J$. Let $\sK_J$ be the product of the first congruence subgroups of ${\rm SL}_{B^\flat_{\fq^\flat}}(T_{\fq^\flat}\psi^\flat)$ as $\fq^\flat$ runs over $J$. The group $N\cap \sK_J$ is a normal pro-$p$ subgroup of $N$. Its image $u(N\cap \sK_J)$ is a normal $p$-subgroup of ${\rm PSL}_{\kappa(\fq_0^\flat)}(\psi^\flat[\fq_0^\flat])$. Since $\fq_0^\flat\notin J$, the latter group is nonabelian simple, so $u(N\cap \sK_J)=1$. Therefore $u$ induces a surjection
\[
\bar u:N/(N\cap\sK_J)\twoheadrightarrow {\rm PSL}_{\kappa(\fq_0^\flat)}(\psi^\flat[\fq_0^\flat]).
\]
As $$N/(N\cap \sK_J)\leq G_J/\sK_J=\prod\limits_{\fq^\flat\in J}{\rm SL}_{\kappa(\fq^\flat)}(\psi^\flat[\fq^\flat]),$$ the group ${\rm PSL}_{\kappa(\fq_0^\flat)}(\psi^\flat[\fq_0^\flat])$ is a subquotient of $\prod\limits_{\fq^\flat\in J}{\rm SL}_{\kappa(\fq^\flat)}(\psi^\flat[\fq^\flat])$. In particular, $|{\rm PSL}_{\kappa(\fq_0^\flat)}(\psi^\flat[\fq_0^\flat])|$ divides $|\prod\limits_{\fq^\flat\in J}{\rm SL}_{\kappa(\fq^\flat)}(\psi^\flat[\fq^\flat])|$. This bounds $|\kappa(\fq_0^\flat)|$, so $\fq_0^\flat$ has only finitely many choices outside $J$. Hence $(G,H)$ also satisfies condition (4) in Theorem~\ref{thm:mainkey}. This completes the proof.
\end{proof}

      Finally, as a consequence, we prove Theorem~\ref{Cor: torsion_finite}. 

\begin{proof}[Proof of Theorem~\ref{Cor: torsion_finite}]
By replacing $\phi$ and $\phi'$ with isogenous Drinfeld modules and passing to a finite extension of $K$, we may assume Simplifications~1--3 from Section~\ref{Sect: simplifications}. These reductions preserve the desired conclusion: torsion finiteness is invariant under isogeny, $K$-isogenous Drinfeld modules have the same full torsion field, and finiteness after a finite extension of the base field implies finiteness over the original field.

Set
$
K^{\rm ab}_{\phi'}:=K^{\rm ab}K_{\phi'}$ and $ \Gamma:=\mathrm{Gal}(K^{\rm sep}/K^{\rm ab}_{\phi'}).
$
Since $K_{\phi'}$ is the fixed field of $\ker(\rho')$, we have
\[
\Gamma=\Gamma_K^{\rm der}\cap\ker(\rho').
\]

By \eqref{Eqn: U_der_CM}, \eqref{Eqn: U_der}, and \eqref{Eqn: main_goal} in the proof of Theorem~\ref{thmmain}, the subgroup $(\rho,\rho')(\Gamma_K^{\rm der})$ is open in $N$, regardless of whether $\phi'$ has potential complex multiplication. Hence
\[
\rho(\Gamma)\times\{1\}=(\rho,\rho')(\Gamma_K^{\rm der})\cap\left[\left(\prod_{\fq\in{\rm Max}\,B}{\rm SL}_{B_\fq}(T_\fq\psi)\right)\times\{1\}\right]
\]
is open in $\left(\prod_{\fq\in{\rm Max}\,B}{\rm SL}_{B_\fq}(T_\fq\psi)\right)\times\{1\}$. It follows that $\rho(\Gamma)$ is open in $\prod_{\fq\in{\rm Max}\,B}{\rm SL}_{B_\fq}(T_\fq\psi)$.

Because $\phi=\psi|_A$ and $B$ is finite over $A$, the full torsion submodules of $\phi$ and $\psi$ coincide, i.e., 
\[
\phi(K^{\rm ab}_{\phi'})_{\rm tors}=\psi(K^{\rm ab}_{\phi'})_{\rm tors}=\bigoplus_{\fq\in{\rm Max}\,B}\psi[\fq^\infty]^\Gamma.
\]
It therefore suffices to prove the following two assertions:
\begin{enumerate}
\item For every $\fq\in{\rm Max}\,B$, the group $\psi[\fq^\infty]^\Gamma$ is finite.
\item For all but finitely many $\fq\in{\rm Max}\,B$, the group $\psi[\fq]^\Gamma$ is trivial.
\end{enumerate}

We first prove~(1). Fix $\fq\in{\rm Max}\,B$, and let $\pi_\fq$ be a uniformizer of $B_\fq$. Since $\rho(\Gamma)$ is open in $\prod\limits_{\fq\in{\rm Max}\,B}{\rm SL}_{B_\fq}(T_\fq\psi)$, its projection $\rho_\fq(\Gamma)$ contains a principal congruence subgroup
\[
K_\fq(r_\fq):=\ker\!\left({\rm SL}_{B_\fq}(T_\fq\psi)\longrightarrow{\rm SL}_{B_\fq/\fq^{r_\fq}B_\fq}(T_\fq\psi/\fq^{r_\fq}T_\fq\psi)\right)
\]
for some integer $r_\fq\geq1$. Since $\phi$ does not have potential complex multiplication, $\psi$ has rank $m=n/[E:F]\geq2$. Choosing a basis of $T_\fq\psi$, we obtain a $\Gamma$-equivariant identification
\[
\psi[\fq^\infty]\simeq V_\fq\psi/T_\fq\psi\simeq(E_\fq/B_\fq)^m.
\]
Consequently,
\[
\psi[\fq^\infty]^\Gamma\subseteq\psi[\fq^{r_\fq}],
\]
which is finite. This proves~(1).

To prove~(2), observe that an open subgroup of $\prod\limits_{\fq\in{\rm Max}\,B}{\rm SL}_{B_\fq}(T_\fq\psi)$ contains the full factor ${\rm SL}_{B_\fq}(T_\fq\psi)$ for all but finitely many $\fq$. Thus, for all such $\fq$, the image of $\Gamma$ on $\psi[\fq]$ contains ${\rm SL}_{\kappa(\fq)}(\psi[\fq])$. Since $m\geq2$, the standard representation of this group has no nonzero fixed vector. Therefore
\[
\psi[\fq]^\Gamma=0 \qquad\text{for all but finitely many }\fq.
\]
This proves~(2), and hence $\phi(K^{\rm ab}_{\phi'})_{\rm tors}$ is finite.
\end{proof}

    \appendix

    \section{Coincidence of morphisms of schemes(by Xuanyou Li)}
    In this section, we prove some results used in the proof of Theorem \ref{11}. We refer \cite[Appendix B]{P} for the detail of Dirichlet density.
    \begin{lem}\label{lemA1}
      Let $f:X\to Y$ be a dominant morphism of integral schemes over $\F_p$ of finite type. For any subset $T$ of closed points of $Y$ with Dirichlet density 0, then so is $f^{-1}(T)\cap|X|$.
    \end{lem}
    \begin{proof}
      Let $d_X=\dim(X)$, $d_Y=\dim(Y)$ and let ${\rm N}(x)$ be cardinality of the residue field $\kappa(x)$ of any closed point $x$ in a scheme. For any subset $W$ of closed points in a scheme, define
      $$F_W(s)=\sum_{x\in W}\frac1{{\rm N}(x)^s}\quad(s\in\mathbb R).$$

      By \cite[Exercise 3.22]{Ha}, shrinking $X$ by an open subscheme if necessary,  we may assume that each fiber of $f$ at $y\in f(X)$ has dimension $ d_X-d_Y$. Then by \cite[Proposition B.1]{P} there exists a constant $c_1>0$ depending only on $f$, such that for any $y\in |Y|$ and all $i\geq1$, we have
      $${\rm card}\Big\{x\in |X|\,\Big|\,f(x)=y\hbox{ and }[\kappa(x):\kappa(y)]=i\Big\}\leq\frac {c_1}i{\rm N}(y)^{i(d_X-d_Y)}.$$
      Hence for any $\delta>0$, we have
      \begin{eqnarray*}
        \sum_{\substack{x\in |X|\\f(x)\in T}}\frac1{{\rm N}(x)^{d_X+\delta}}=\sum_{y\in T}\sum_{\substack{x\in |X|\\f(x)=y}}\frac1{{\rm N}(x)^{d_X+\delta}}\leq\sum_{y\in T}\sum_{i=1}^\infty\frac{c_1{\rm N}(y)^{i(d_X-d_Y)}}{i{\rm N}(y)^{i(d_X+\delta)}}\leq\sum_{y\in T}\sum_{i=1}^\infty\frac {c_1}{{\rm N}(y)^{i(d_Y+\delta)}}\leq \sum_{y\in T}\frac{2c_1}{{\rm N}(y)^{d_Y+\delta}}.
      \end{eqnarray*}
      So $F_{f^{-1}(T)\cap|X|}(d_X+\delta)\leq 2c_1F_T(d_Y+\delta)$.

      Let $q_X$ and $q_Y$ be the cardinality of the field of constant in the function fields of $X$ and $Y$, respectively.  By the proof of \cite[Proposition B. 5\; (b)]{P}, there exists $c_2,c_3>0$ such that for any $\delta>0$, we have
      \begin{eqnarray*}
        &&F_{|Y|}(d_Y+\delta)\leq -c_2\log(1-q_Y^{-\delta})\\
        &&F_{|X|}(d_X+\delta)\geq-\log(1-q_X^{-\delta})+c_3\log(1-q_X^{-\delta-\frac12}).
      \end{eqnarray*}
      This implies that
      \begin{eqnarray*}
        \frac{F_{f^{-1}(T)\cap|X|}(d_X+\delta)}{F_{|X|}(d_X+\delta)}&\leq &2c_1\frac{F_T(d_Y+\delta)}{F_{|Y|}(d_Y+\delta)}\cdot\frac{F_{|Y|}(d_Y+\delta)}{F_{|X|}(d_X+\delta)}\\&\leq&2c_1\frac{F_T(d_Y+\delta)}{F_{|Y|}(d_Y+\delta)}\cdot\frac{-c_2\log(1-q_Y^{-\delta})}{-\log(1-q_X^{-\delta})+c_3\log(1-q_X^{-\delta-\frac12})}.
      \end{eqnarray*}
      In other words, $f^{-1}(T)\cap|X|$ also has zero density in $X$.
    \end{proof}

    \begin{thm}\label{thmf=g}
      Let $f_1,f_2:X\to Y$ be two dominant morphisms of integral $\F_p$-schemes of finite type with $Y$ separated. If $S\coloneqq\{x\in|X|\,\big|\, f_1(x)=f_2(x)\}$ has positive Dirichlet density in $X$, then there exists a natural number $d$ such that $f_2=f_1\circ {\rm Fr}_X^d$ or $f_1=f_2\circ{\rm Fr}_X^d$, where ${\rm Fr}_X:X\to X$ denotes the morphism associated to the $p$-th power map on $\mathcal O_X$.
    \end{thm}
    \begin{proof}
      First, assume $X$ and $Y$ are curves over $\F_p$. Since $Y$ is separated, by shrinking $X$ and $Y$ we may assume that they are affine smooth curves over $\F_p$. Replacing $X$ and $Y$ by its completion if necessary, we may assume that $X$ and $Y$ are projective smooth curves over $\F_p$.

      For any $d\in\Z$, let $Z_d$ be the fiber product of $(f_1,f_2\circ{\rm Fr}_X^d):X\to Y\times_{\F_p}Y$ (resp. $(f_1\circ{\rm Fr}_X^{-d},f_2):X\to Y\times_{\F_p}Y$) and the diagonal morphism $Y\to Y\times_{\F_p}Y$ if $d\geq0$ (resp. if $d<0$). Since the scheme $Y$ is separated, then each $Z_d$ is a closed subscheme of $X$. To prove Step 1, it suffices to show $Z_d=X$ for some $d$.

      Assume $Z_d\neq X$ for any $d\in\Z$. Then $\dim (Z_d)=0$ and we can define $\deg(Z_d)=\dim_{\F_p}(\mathcal O(Z_d))$. We claim that there exists a constant $c'>0$ such that $\deg(Z_d)\leq c'p^{|d|}$ for any $d\in\Z$.

      We prove this claim by $l$-adic cohomology and Deligne's Weil II, where $l\neq p$ is an another prime integer. For any object $D$ over $\F_p$, denote by $\overline D$ its base change to the algebraic closure of $\F_p$.  Without loss of generality, assume that $d\geq0$. Then
      $$\deg(Z_d)={\rm Tr}((\bar f_1,\bar f_2\circ\overline{\rm Fr}_{ X}^d)^{*}([\Delta_{\overline Y}])),$$
      where ${\rm Tr}$ is the trace map
      $${\rm Tr}:H^2(\overline X,\overline{\mathbb Q}_l(1))\to\overline{\mathbb Q}_l,$$
      $\Delta_{\overline Y}$ is the schematic image of the diagonal morphism $\overline Y\to\overline Y\times_{\overline\F_p}\overline Y$, and where
      $[\Delta_{\overline Y}]$ is the cycle class of $\Delta_{\overline Y}$ in $\overline Y\times_{\overline\F_p}\overline Y$ given by
      $$[\Delta_{\overline Y}]\in H^2(\overline Y \times_{\overline\F_p}\overline Y ,\overline{\mathbb Q}_l(1))=\bigoplus_{i=0}^2 H^i(\overline Y ,\overline{\mathbb Q}_l(1))\otimes H^{2-i}(\overline Y,\overline{\mathbb Q}_l).$$
      Choose a basis $\{t_{i1},\ldots,t_{ib_i}\}$ of $H^i(\overline Y,\overline{\mathbb Q}_l)$ under which
      the matrix of the action ${\rm Fr}_{\overline X}^*$ is of a Jordan form.
      By Deligne's Weil II,  any complex conjugate of each eigenvalue of the action of $\overline{\rm Fr}_{ Y}^{*}$ on $H^i(\overline Y,\overline{\mathbb Q}_l)$ has absolute values $\leq p^{\frac{i}2}$. Then there exists a constant $c>0$ such that
      $${\rm Tr}((\bar f_1,\bar f_2\circ\overline{\rm Fr}_{ X}^d)^{*}([\Delta_{\overline Y}]))<cp^d\quad(d\geq0).$$
      This proves the claim.

      For any positive integer $m$ and integer $d$, let
      $$N(m,d)={\rm card}\{x\in|Z_d|\,\big|\, {\rm N}(x)=p^m\}.$$
      It follows immediately from the claim that $\sum\limits_{m=1}^\infty mN(m,d)<cp^{|d|}$. There exists $c'>0$ such that
      \begin{eqnarray*}
        \sum_{x\in S}\frac1{{\rm N}(x)}=\sum_{m=1}^\infty\sum_{d=1-\lceil{\frac m2}\rceil}^{{\lfloor{\frac m2}\rfloor}}\frac{N(m,d)}{p^m}=c'+\sum_{d\in\Z\backslash\{0\}}\sum_{m\geq|2d|}\frac{N(m,d)}{p^m}\leq c'+\sum_{d\in\Z\backslash\{0\}}\frac{cp^{|d|}/|2d|}{p^{|2d|}}=c'+\sum_{d=1}^\infty\frac {c}{dp^d}<+\infty.
      \end{eqnarray*}
      So $S$ has Dirichlet density zero in $X$, which gets a contradiction. This proves the theorem for curves.

      For general $X$ and $Y$, by shrinking them if necessary, we may assume that $X$ and $Y$ are affine. Note that the theorem obviously holds if $\dim(Y)=0$. From now on, assume $\dim(Y)\geq1$. For each $i$, the morphism $f_i:X\to Y$ defines a homomorphisms $a\mapsto a_i:\;\mathcal O_Y(Y)\to\mathcal O_X(X)$. Denote $\F_p[a]$ the subring of $\mathcal O_Y(Y)$ generated by $a$, and by $\F_p[a_1,a_2]$ the subring of $\mathcal O_X(X)$ generated by $a_1$ and $a_2$. We thus have a commutative diagram
      \[\xymatrix{X\ar[rr]^{\widetilde\pi}\ar[d]^{f_i}&&\Spec\;\F_p[a_1,a_2]\ar[d]^{p_i}\\Y\ar[rr]^\pi&&\Spec\;\F_p[a],}\]
      where $p_i$ is the morphism given by $a\mapsto a_i$. Applying Lemma \ref{lemA1} to the dominant morphism $\widetilde\pi$, we have $\widetilde\pi(S)$ has positive density in $\Spec\;\F_p[a_1,a_2]$.


      Suppose $a_1$ and $a_2$ are algebraic independent over $\F_p$. Then $\Spec\;\F_p[a_1,a_2]\simeq\A^2_{\F_p}$ and $\Spec\;\F_p[a]\simeq\A^1_{\F_p}$. Under these identifications, $p_i$ is the $i$-th projection $\A^2_{\F_p}\to\A^1_{\F_p}$. For any $z\in\widetilde\pi(S)$, we have $p_1(z)=p_2(z)$ and hence $\widetilde\pi(S)\subset\coprod\limits_{m=1}^\infty S_m$ with
      $$S_m=\big\{z\in|\A^2_{\F_p}|\,\big|\,p_1(z)=p_2(z)\hbox{ and } {\rm N}(z)=p^m\big\}.$$
      Any $z\in S_m$ is defined by an element $(u,u^{p^i})\in\overline \F_p^2$ such that $\F_{p^m}=\F_p(u)$ and $1\leq i\leq m$. Two such elements $(u,u^{p^i})$ and $(v,v^{p^{j}})$ correspond the same $z\in S_m$ if $i=j$ and $v=u^{p^e}$ for some $1\leq e\leq m$. So ${\rm card}(S_m)\leq\frac{p^m}{m}\cdot m=p^m$. Hence
      $$\sum_{z\in\widetilde\pi(S) }\frac1{{\rm N}(z)^2}\leq\sum_{m=1}^\infty\sum_{z\in S_m}\frac1{{\rm N}(z)^2}\leq\sum_{m=1}^\infty\frac{{\rm card}(S_m)}{p^{2m}}\leq\sum_{m=1}^\infty\frac1{p^m}<+\infty.$$
      This shows $\widetilde\pi(S) $ has zero density in $\A^2_{\F_p}$ which get a contradiction. So $a_1$ and $a_2$ are algebraic dependent over $\F_p$.

      Fix an element $a\in\mathcal O_Y(Y)$ transcendental over $\F_p$. Then $p_1$ and $p_2$ are dominant morphism of integral curves. We have shown that there exists a unique natural number $e_a$ such that $a_2=a_1^{p^{e_a}}$ or $a_1=a_2^{p^{e_a}}$. In other words, there exist a unique $d_a\in\Z$ such that $a_2=a_1^{p^{d_a}}$ in an algebraic closure $\overline{K(X)}$ of the function field $K(X)$ of $X$. To prove the theorem, it suffices to show  $b_2=b_1^{p^{d_a}}$ for any $b\in\mathcal O_Y(Y)$. It holds if $b$ lies in the subfield $\F_p(a)$ of $K(Y)$ generated by $a$, and it holds for $b$ if it holds for $a+b$. Since at least one element of $b$ and $a+b$ is transcendental, we may assume that $b$ is transcendental over $\F_p$ and $b\notin\F_p(a)$. Then we also have $d_b\in\Z$ such that $b_2=b_1^{p^{d_b}}$.

      First suppose $a$ and $b$ are algebraic dependent over $\F_p$. Choose $h_j\in\F_p(a)$ such that $\sum\limits_{j=0}^kh_j(a)b^j=0$ with $k$ as small as possible and $h_k=1$. As $b\notin\F_p(a)$, we have $k\geq2$. Without loss of generality, assume $d_a\geq d_b$. Under the homomorphism $\mathcal O_Y(Y)\hookrightarrow\mathcal O_X(X)$ induced by $f_i$, we have
      $$0=\sum\limits_{j=0}^kh_j(a_1)b_1^j=\sum\limits_{j=0}^kh_j(a_2)b_2^j=\sum_{j=0}^kh_j(a_1^{p^{d_a}})b_1^{jp^{d_b}}=\sum_{j=0}^kh_j(a_1^{p^{d_a-d_b}})b_1^{j}.$$
      By the minimality of the choice of $k$, we have $h_j(a_1^{p^{d_a-d_b}})=h_j(a_1)$. The transcendentality of $a_1$ implies at least one $h_j\notin \F_p$, which implies $d_a=d_b$.

      Suppose now $a$ and $b$ are algebraic independent over $\F_p$. Then $a+b$ is transcendental over $\F_p$, and there also exists $d_{a+b}\in\Z$ such that $a_2+b_2=(a_1+b_1)^{p^{d_{a+b}}}$. Then the algebraic relation $$a_1^{p^{d_a}}+b_1^{p^{d_{b}}}=a_2+b_2=(a_1+b_1)^{p^{d_{a+b}}}=a_1^{p^{d_{a+b}}}+b_1^{p^{d_{a+b}}}\in\overline{K(X)}$$
      implies $d_a=d_b=d_{a+b}$. Hence $b_2=b_1^{p^{d_a}}$.
    \end{proof}

    \section*{Acknowledgements} The authors thank Yang Cao for discussing with the norm maps of idele rings of global fields. The second author was supported by the National Key R\&D Program of China No. 2023YFA1009702 and the National Natural Science Foundation of China (grant no. 12271371)
    \bibliographystyle{plain}

    \end{document}